\documentclass[12pt]{article}

\usepackage[english]{babel}

\usepackage[letterpaper,top=2cm,bottom=2cm,left=3cm,right=3cm,marginparwidth=1.75cm]{geometry}

\usepackage{amsmath}
\usepackage{graphicx}
\usepackage[colorlinks=true, allcolors=blue]{hyperref}
\usepackage{amsfonts}
\usepackage{amssymb}
\usepackage{mathrsfs}
\usepackage{mathtools}
\usepackage{hyperref}
\usepackage{color,xcolor}
\usepackage[dvipsnames]{xcolor}
\usepackage{graphicx}
\usepackage{indentfirst}
\usepackage{amsmath}
\usepackage{amsfonts}
\usepackage{amssymb}
\usepackage{mathrsfs}
\usepackage{mathtools}
\usepackage{hyperref}
\usepackage{color,xcolor}
\usepackage{graphicx}
\usepackage{amsmath}
\usepackage{pifont}
\usepackage{ntheorem}
\usepackage{dsfont}
\usepackage{ntheorem}
\DeclareMathOperator*{\esssup}{ess\,sup}
\allowdisplaybreaks[2]

\newcommand{\enabstractname}{Abstract}

\newenvironment{enabstract}{%
      \par
      \noindent\mbox{}\hfill{\bfseries\enabstractname}\hfill\mbox{}\par
      \vskip 2.5ex}{\par\vskip 2.5ex}
\title{Uniqueness for nonlinear Fokker-Planck equations with general diffusion terms and their associated nonlinear Markov processes}
\author{Viorel Barbu \thanks{Al.I. Cuza University and Octav Mayer Institute of Mathematics of Romanian Academy, Ia\c{s}i, Romania}
\and Yuqi Li \thanks{Faculty of Mathematics, Bielefeld University, D-33501 Bielefeld, Germany}
\and Michael Röckner$^{\dag}$ \thanks{Academy for Mathematics and Systems Science, CAS, Beijing, China and School of Data Science, The Chinese University of Hong Kong, Shenzhen (CUHK-Shenzhen), China} }

\begin{document}
\maketitle

\newtheorem{remark}{Remark}[section]
    \newtheorem{definition}{Definition}[section]
    \newtheorem{lemma}{Lemma}[section]
    \newtheorem{corollary}{Corollary}[section]
    \newtheorem{claim}{Claim}[section]
    \newtheorem{theorem}{Theorem}[section]
    \newtheorem{proposition}{Proposition}[section]
    \newtheorem*{proof}{Proof.}
    \newtheorem{example}{Example}[section]

\begin{enabstract}
    This work is concerned with the uniqueness of distributional solutions to nonlinear Fokker-Planck equations with non-diagonal diffusion terms of type
\begin{equation}
   u_{t}-\sum_{i,j=1}^{d} D^{2}_{ij}(a_{ij}(x)\beta(x,u))+ \text{div}(b(x,u)u)=0  \quad \text{in}\;  (0, \infty) \times \mathbb{R}^{d} ,\notag
\end{equation}
with initial condition $u(0,x)\equiv u_{0}(x)$, where $a_{ij}$, $\beta$, and $b$ are suitable functions. Under suitable assumptions, this equation generates a continuous contraction semigroup $S(t): L^{1}(\mathbb{R}^{d}) \rightarrow L^{1}(\mathbb{R}^{d})$, and $u(t)=S(t)u_{0}$ is a mild solution to the equation. Our main contribution is to prove that this mild solution is unique in the much larger class of distributional solutions. This extends previous uniqueness results for the diagonal (also called isotropic) diffusion case $a_{ij} \equiv \delta_{ij}$.
    Another key analytical result of this paper is the uniqueness for distributional solutions of the associated linearized equation. As a main application, we prove weak uniqueness for the corresponding McKean-Vlasov SDEs. Moreover, we prove that, the probabilistically weak solution to the McKean-Vlasov SDEs is also the unique probabilistically strong solution. Furthermore, we establish a new $L^{\infty}$ estimate for mild solutions starting from data in $L^{1}\cap L^{\infty}$ and this estimate is used in the construction of nonlinear Markov processes. Finally, we prove that the path laws of the solutions to the McKean-Vlasov SDEs form a nonlinear Markov process in the sense of McKean. 
\end{enabstract}

\section{Introduction}
We study nonlinear Fokker-Planck equations of the form
\begin{equation}\label{*}
\begin{split}
    &\frac{\partial}{\partial t}u(t,x)-\sum_{i,j=1}^{d}D_{ij}^{2}(a_{ij}(x, u(t,x))u(t,x))+\text{div}(b(x,u(t,x))u(t,x))=0,\\
    & \quad \quad \quad \quad \quad \quad \quad \quad \quad \quad \quad \quad \quad \quad \quad\quad \quad \quad \quad \quad\quad \quad  t >0,\; x \in \mathbb{R}^{d}, \\ &u(0,x)=u_{0}(x), \; x \in \mathbb{R}^{d},
\end{split} 
\end{equation}
for diffusion terms of the type $a_{ij}(x,u)=a_{ij}(x)\beta(x,u)/u$. Thus the equation considered in this paper is
\begin{equation}\label{1}
\begin{split}
    &\frac{\partial}{\partial t}u(t,x)-\sum_{i,j=1}^{d}D_{ij}^{2}(a_{ij}(x)\beta(x,u(t,x)))+\text{div}(b(x,u(t,x))u(t,x))=0,\\
    & \quad \quad \quad \quad \quad \quad \quad \quad \quad \quad \quad \quad \quad \quad \quad\quad \quad \quad \quad \quad\quad \quad  t >0,\; x \in \mathbb{R}^{d}, \\ &u(0,x)=u_{0}(x), \; x \in \mathbb{R}^{d},
\end{split} 
\end{equation}
where $a_{ij}:\mathbb{R}^{d}\rightarrow \mathbb{R}$, $\beta :\mathbb{R}^{d}\times \mathbb{R}\rightarrow \mathbb{R}$, $b:\mathbb{R}^{d} \times \mathbb{R}\rightarrow \mathbb{R}^{d}$ are given functions. We study the existence of mild solutions (see Definition \ref{def2.1} below) and, more importantly, the uniqueness of distributional solutions (see Definition \ref{def2.2} below) for (\ref{1}) under the following hypotheses on the coefficients $a_{ij}$, $\beta$, and $b$.\\
\noindent \textcolor{black}
{\textbf{Hypothesis (\MakeUppercase{\romannumeral 1})}}\\
(\textit{\textbf{\romannumeral1}}) $\beta \in C^{1}(\mathbb{R}^{d}\times \mathbb{R})$, $\beta_{r}(x,r)\geqslant 0$ for all $x \in \mathbb{R}^{d}$ and $r \in \mathbb{R}$; $\beta (x,0)\equiv0$ for all $x \in \mathbb{R}^{d}$; $\beta_{r}\in L^{\infty}(\mathbb{R}^{d}\times (-N,N))$ for every $N >0$.\\
\noindent (\textcolor{black}{\textit{\textbf{\romannumeral2}}}) $a_{ij} \in C^{\infty}(\mathbb{R}^{d})\cap C_{b}(\mathbb{R}^{d})$, $(a_{ij})_{x}  \in C_{b}(\mathbb{R}^{d};\mathbb{R}^{d})$, $a_{ij}=a_{ji}$, and $\exists \gamma >0$ such that
\begin{equation}
    \sum\limits_{i,j=1}^{d}a_{ij}(x)\xi_{i}\xi_{j} \geqslant \gamma \left | \xi \right |^{2}, \quad \forall \xi \in \mathbb{R}^{d}, x \in \mathbb{R}^{d}. \notag  
\end{equation} 
\noindent (\textcolor{black}{\textit{\textbf{\romannumeral3}}}) $ b \in C_{b}(\mathbb{R}^{d+1};\mathbb{R}^{d})$ and $b(x, \cdot) \in C^{1}(\mathbb{R};\mathbb{R}^{d})$ for all $x \in \mathbb{R}^{d}$.  \\
\noindent (\textcolor{black}{\textit{\textbf{\romannumeral4}}}) For each compact $K \subseteq \mathbb{R}$, $\exists \alpha_{K}\in (0,\infty)$ such that 
\begin{equation}
    \left | b(x, r_{1})r_{1}-b(x, r_{2})r_{2} \right | \leqslant \alpha_{K}\left | \beta (x,r_{1})-\beta (x,r_{2}) \right |, \quad \forall r_{1},r_{2} \in K, \forall x \in \mathbb{R}^{d}.\notag 
\end{equation}
Here, $b(x,u)=\left\{b^{i}(x,u) \right\}^{d}_{i=1}$, $(a_{ij})_{x}=\nabla_{x}a_{ij}$. 
\begin{remark}
    \rm If $\displaystyle \sup_{x\in\mathbb{R}^{d}, r\in K }\left| b_{r}(x,r) \right| < \infty $, $\displaystyle \inf_{x\in\mathbb{R}^{d}, r\in K}\beta_{r}(x,r)>0$ for every compact $K \subseteq \mathbb{R}$, where $b_{r}(x,r)=\displaystyle \frac{\partial}{\partial r}b(x,r)$ and $\displaystyle \beta_{r}(x,r)=\frac{\partial}{\partial r}\beta(x,r)$, then condition (\textit{\textbf{\romannumeral4}}) can be omitted, because it then follows from (\textit{\textbf{\romannumeral3}}).
\end{remark}

Nonlinear Fokker-Planck equations (abbreviated NFPEs) arise naturally in mean-field game theory and statistical mechanics and have also found numerous applications in areas such as information theory, graph theory, and economics (see, for instance, \cite{ref16}, \cite{ref18}, \cite{ref19} and the references therein). Beyond these applications, nonlinear Fokker-Planck equations provide a powerful framework for the analysis of nonlinear diffusion phenomena and stochastic processes. A particularly important characteristic of nonlinear Fokker-Planck equations is their connection with McKean-Vlasov stochastic differential equations (also called distribution-dependent stochastic differential equations, abbreviated DDSDEs or McKean-Vlasov SDEs), whose coefficients depend not only on the solution process, but also on the one-dimensional time-marginal law of the solution. More precisely, (\ref{1}) can be used to solve the following McKean-Vlasov stochastic differential equation
\begin{equation}\label{9}
\begin{aligned}
   & dX(t)=b(X(t), u(t, X(t)))dt+ \displaystyle \sqrt{\frac{2\beta(X(t),u(t,X(t)))}{u(t,X(t))}}\sigma(X(t))dW(t),\forall t>0, \\
 & \mathcal{L}_{X(0)}(dx)=u_{0}(x)dx,\quad \mathcal{L}_{X(t)}(dx)=u(t,x)dx,\quad t>0, 
\end{aligned}    
\end{equation}
on a probability space $(\Omega ,\mathscr{F},\mathbb{P})$ with normal filtration $ (\mathscr{F}_{t})_{t \geqslant 0}$ and an $(\mathscr{F}_{t})$-Brownian motion $W(t)$ on it with values in $\mathbb{R}^{d}$. Here $\sigma=(\sigma_{ij})_{1 \leqslant i,j \leqslant d}$ with $(\sigma\sigma^{T})_{ij}=a_{ij}$, and $\mathcal{L}_{X_{t}}$ denotes the law of $X_{t}$ under $\mathbb{P}$, and $u_{0} \in \mathcal{P}$ is given.\\
\indent Equation (\ref{1}) extends the isotropic diffusion model, i.e. where $a_{ij}\equiv \delta_{ij}$, $\beta(x,u)\equiv \beta(u)$, and $b(x,u)=D(x)b(u)$, namely
\begin{equation}\label{1.}
  \begin{split}
 &\frac{\partial}{\partial t}u(t,x)-\Delta \beta (u(t,x))+ \text{div}(D(x)b(u(t,x))u(t,x))=0,\quad t> 0,x\in \mathbb{R}^{d}, \\ &
 \; u(0,x)=u_{0}(x), \quad x\in \mathbb{R}^{d}.
\end{split} 
\end{equation}
This isotropic case can also be viewed as an extension of the classical Smoluchowski equation and has been studied extensively by V. Barbu and M. R\"ockner. In \cite{ref7}, they proved the existence of mild solutions to (\ref{1.}) and of weak solutions to the associated McKean-Vlasov SDEs. They also proved existence for measures of bounded variation as initial data and established a smoothing effect for the initial data. Building on this existence theory, they further developed the uniqueness theory, in particular, they proved the existence and uniqueness of mild solutions and the uniqueness of distributional solutions in \cite{ref6, ref7, ref11}, as well as the uniqueness of weak solutions to the corresponding McKean-Vlasov SDEs. Additionally, they considered equation (\ref{1.}) with the fractional Laplacian $(-\Delta)^{s}$, $0<s<1$, in the diffusion term, namely $(-\Delta)^{s} \beta (u(t,x))$, and developed the existence and uniqueness theory for nonlinear Fokker-Planck equations and McKean-Vlasov SDEs with L\'evy noise \cite{ref12}. Recently, in \cite{ref14}, V. Barbu extended the well-posedness results for (\ref{1.}) (and for the associated McKean-Vlasov SDEs) to a more general class of nonlinear Fokker-Planck equations with singular integral drifts $\text{div}((D(x)b(u) + K * u)u)) $, where $K: \mathbb{R}^{d} \rightarrow \mathbb{R}^{d}$ is a kernel and $*$ denotes convolution. For nonlinear Fokker-Planck equations with time-dependent coefficients, to the best of our knowledge, there is only one related reference, \cite{ref9}. More specifically, V. Barbu and M. R\"ockner studied (\ref{1.}) with coefficients $\beta \equiv a(t,x,u)u$ and $Db \equiv b(t,x,u)$ and obtained existence results in appropriate Hilbert spaces. Finally, it should be mentioned that, for porous media equations (that is, $D \equiv 0$ in (\ref{1.})), uniqueness of distributional solutions was first proved by H. Brezis and M.G. Crandall \cite{ref17} and also by M. Pierre \cite{ref26}, similar results were established in \cite{ref5}, \cite{ref6} and \cite{ref27}.\\
\indent For the general diffusion terms $a_{ij}(x,u)$, the existence of mild solutions to (\ref{1}) has been obtained in \cite{ref4} via nonlinear semigroup theory in $L^{1}(\mathbb{R}^{d})$, while the uniqueness of distributional solutions to (\ref{1}) remained open. Thus, one of the main contributions of this paper is to prove uniqueness of distributional solutions to NFPE (\ref{1}), see Theorem \ref{thm3.3} and Corollary \ref{col1} below. These results are not covered by previous work. In addition, we would like to point out that we also treat the general drift term $b(x,u)$ (i.e., the case with explicit $x$-dependence). Providing a supplementary proof based on \cite{ref4}, we prove an $L^{\infty}$ estimate for mild solutions to (\ref{1}) starting from initial data in $L^{1}\cap L^{\infty}$. This estimate is essential for the solution class used subsequently, especially in Section 5, where nonlinear Markov processes are constructed, see Theorem \ref{thm3.2} and Corollary \ref{col3.2} for details. An additional major analytical result is the proof of ``linearized uniqueness'' for distributional solutions to (\ref{1}), see Theorem \ref{thm4.1} and Corollary \ref{col2}. \\
\indent Furthermore, we prove uniqueness of probabilistically weak solutions to the corresponding McKean-Vlasov SDE (\ref{9}) (see Theorem \ref{thm4.2} below), which is another main contribution of this paper. For related work on weak uniqueness of McKean-Vlasov SDEs, we refer to \cite{ref5}, \cite{ref6}, \cite{ref15}, \cite{ref21} and \cite{ref23}, but none of these papers covers the results in Theorem \ref{thm4.2}. Moreover, we prove that, the probabilistically weak solution $X(t)$ to the McKean-Vlasov SDE (\ref{9}) is also the unique probabilistically strong solution, see Theorem \ref{thm4.3} for details. As another important consequence, we also prove that the path laws of the solutions to McKean-Vlasov SDEs form a nonlinear Markov process in the sense of McKean (see \cite{ref29}). This concept of a nonlinear Markov process was recently elaborated by M. Rehmeier and M. R\"ockner in \cite{ref28}. The results on nonlinear Markov processes in our case are presented and proved in Section 5.\\

\noindent \textbf{Notation.} By $C^{k}(\mathbb{R}^{d})$ ($C^{k}(\mathbb{R})$, $C^{k}(\mathbb{R}^{d} \times \mathbb{R})$, respectively) we denote the space of $k$-times continuously differentiable functions on $\mathbb{R}^{d}$ ($\mathbb{R}$, $\mathbb{R}^{d} \times \mathbb{R}$, respectively); by $C_{b}(\mathbb{R}^{d})$ ($C_{b}(\mathbb{R})$, $C_{b}(\mathbb{R}^{d} \times \mathbb{R})$, respectively) we denote the space of continuous and bounded functions on $\mathbb{R}^{d}$ ($\mathbb{R}$, $\mathbb{R}^{d} \times \mathbb{R}$, respectively); and $C([0,T])$ denotes the space of continuous real-valued functions on $[0,T]$. We shall simply write $C^{k}$ for $C^{k}(\mathbb{R}^{d})$ and $C_{b}$ for $C_{b}(\mathbb{R}^{d})$.\\
\indent $L^{p}(\mathbb{R}^{d})=L^{p}$ for $1\leqslant p \leqslant \infty$ denotes the space of all real-valued, $p$-integrable functions on $\mathbb{R}^{d}$, with standard norm $\left| \cdot \right|_{p}$. By $L^{p}_{loc}$ we denote the corresponding local space. $(\cdot ,\cdot)$ shall denote the scalar product in $L^{2}$. Let $H^{k}(\mathbb{R}^{d})=H^{k}$, $k=1,2$, denote the standard Sobolev space on $\mathbb{R}^{d}$ and by $H^{-k}(\mathbb{R}^{d})=H^{-k}$ we denote the dual space of $H^{k}(\mathbb{R}^{d})$.\\
\indent $C^{\infty}_{0}(\mathbb{R}^{d})$ is the space of infinitely differentiable real-valued functions with compact support in $\mathbb{R}^{d}$. $C^{\infty}_{0}([0, \infty) \times \mathbb{R}^{d})$ denotes the space of all $\varphi \in C^{\infty}_{0}([0, \infty) \times \mathbb{R}^{d})$ with compact support in $ [0, \infty) \times \mathbb{R}^{d}$. By ${\mathcal{D}}'(\mathbb{R}^{d})$ (${\mathcal{D}}'((0, \infty) \times \mathbb{R}^{d})$, respectively) we denote the space of Schwartz distributions on $\mathbb{R}^{d}$ ($(0, \infty) \times \mathbb{R}^{d}$, respectively).\\
\indent $\mathcal{P} $ denotes the set of all probability densities on $\mathbb{R}^{d}$, that is
\begin{equation}
    \mathcal{P}= \left \{ \rho \in L^{1}(\mathbb{R}^{d}); \; \rho \geqslant 0, a.e. \text{in}\;  \mathbb{R}^{d}, \; \int_{\mathbb{R}^{d}}\rho (x)dx =1\right \}  .\notag
\end{equation}
$D_{i}u$ denotes the partial derivative of a function $u=u(x_{1}, \cdots, x_{d})$ with respect to $x_{i}$, $1 \leqslant i \leqslant d$, and by $D^{2}_{ij}u$ we denote the second order derivatives $ \frac{\partial^{2}u}{\partial x_{i} \partial x_{j}}$. We also set, for each $u \in C^{1}(\mathbb{R}^{d} \times \mathbb{R})$,
\begin{equation}
    u_{r}=\frac{\partial}{\partial r}u(x,r), \quad u_{x}=\nabla_{x}u(x,r).\notag
\end{equation}
The function $t \rightarrow (u(t,\cdot)dx)_{[0,T]}$ is said to be narrowly continuous if, for every $ \eta \in C_{b}(\mathbb{R}^{d}) $, 
\begin{equation}
    \lim_{t \rightarrow s}\int_{\mathbb{R}^{d}}\eta(x)u(t,x)dx=\int_{\mathbb{R}^{d}}\eta(x)u(s,x)dx, \quad s \geqslant 0.\notag
\end{equation}

\indent The remainder of this paper is structured as follows. Section 2 contains the preliminary definitions and results used throughout the paper. Section 3 contains the existence of mild solutions and the uniqueness of distributional solutions for NFPE (\ref{1}). In Section 4, we prove weak uniqueness for the corresponding McKean-Vlasov SDE (\ref{9}). Section 5 contains the main results on nonlinear Markov processes.

\section{Preliminaries}
\indent We now define mild and distributional solutions to NFPE (\ref{1}), which are the basic notions used in this work.
\begin{definition}\label{def2.1}\rm (mild solution)
  A continuous function $u: [0, \infty) \rightarrow L^{1}$ is called a mild solution to NFPE (\ref{1}) if, for each $T>0$,
\begin{equation}
    u(t)=\displaystyle\lim_{h \rightarrow 0}u_{h}(t) \ \text{in} \ L^{1}, \quad \text{uniformly\; on} \; [0,T], \notag    
\end{equation}
where $u_{h}:[0,T] \rightarrow L^{1}$ is the step function
\begin{equation}
    u_{h}(t)=u_{h}^{k}, \ \forall t \in [kh,(k+1)h), k=0,1, \cdots , N_{h}=[\frac{T}{h}],\notag
\end{equation}
\begin{equation}
    u_{h}^{k+1}+hA(u_{h}^{k+1})=u_{h}^{k}, \ \forall k=0,1, \cdots ,N_{h}; \ u_{h}^{0}=u_{0}. \notag
\end{equation}
Here, $A$ is the operator defined as follows
\begin{equation}
\begin{split}
    &A(u)=-\sum\limits_{i,j=1}^{d} D^{2}_{ij}(a_{ij}(x)\beta (x,u))+ \text{div}(b(x,u)u),\ \forall u \in D(A), \\&
    D(A)=\left\{u \in L^{1}; A(u) \in L^{1}\right\},\notag
\end{split}
\end{equation}
\end{definition}
where $D^{2}_{ij}$ and \text{div} are taken in the sense of Schwartz distributions on $ \mathbb{R}^{d}$.
\begin{definition}\label{def2.2} \rm (1) (distributional solution)
Let $d \geqslant 1$ and $T>0$. A function $u \in L_{\text{loc}}^{1}([0,T)\times \mathbb{R}^{d})$ is said to be a distributional solution to NFPE (\ref{1}) with initial condition $u_{0}$ if $a_{ij}(x)\beta(x,u) \in L_{\text{loc}}^{1}([0,T)\times \mathbb{R}^{d})$, $b(x,u)u \in L_{\text{loc}}^{1}([0,T)\times \mathbb{R}^{d};\mathbb{R}^{d})$, and, for all $ \varphi \in C^{\infty}_{0}([0,T)\times \mathbb{R}^{d})$ and $u_{0} \in L^{1}(\mathbb{R}^{d})$,
\begin{equation}\label{0}
\begin{split}
   \int^{T}_{0} &\int_{\mathbb{R}^{d}} \bigg(u(t,x)\frac{\partial \varphi }{\partial t}(t,x)+\sum\limits_{i,j=1}^{d}a_{ij}(x)\beta(x,u(t,x))D^{2}_{ij}\varphi(t,x)\\+& u(t,x)b(x, u(t,x)) \cdot \nabla \varphi (t,x) \bigg) dxdt +\int_{\mathbb{R}^{d}}u_{0}(x)\varphi(0,x)dx=0.  
\end{split}
\end{equation}
(2) A distributional solution to NFPE (\ref{1}) is called a probability solution to NFPE (\ref{1}) if $u_{0} \in \mathcal{P}(\mathbb{R}^{d})$, and $u(t,\cdot) \in \mathcal{P}(\mathbb{R}^{d})$ for a.e. $t \in (0.T)$. \vspace{0.35em}\\
(3) Let $\mathcal{\tilde P}(\mathbb{R}^{d})$ denote the space of all Borel probability measure on $\mathbb{R}^{d}$. Suppose $ \mathcal{P}_{0} \subset \mathcal{\tilde P}(\mathbb{R}^{d}) $ such that for each $(s, \zeta) \in [0, \infty) \times \mathcal{P}_{0}$, there exists a weakly continuous probability solution $ [s, \infty) \ni t \mapsto \mu_{t}^{s,\zeta} \in \mathcal{P}_{0}$ to NFPE (\ref{1}) with starting time $s \in [0, \infty)$ and initial condition $u_{0}$ renamed as $\zeta$, and $\mu_{t}^{s,\zeta}:=u^{s,\zeta}(t,x)dx$, such that flow property
\begin{equation}
    \mu_{t}^{s, \zeta}=\mu_{t}^{r, \mu_{r}^{s, \zeta}} ,\; \; \forall \; 0 \leqslant s \leqslant r \leqslant t, \;\; \zeta \in \mathcal{P}_{0} \notag
 \end{equation}
holds. Then $ (\mu^{s,\zeta})_{(s,\zeta) \in [0,\infty)\times \mathcal{P}_{0} }$ is called a solution flow of NFPE (\ref{1}) in $ \mathcal{P}_{0} $.

\end{definition}

\begin{remark}\label{rem2.1}
   \rm For an initial condition $u_{0} \in L^{1}$, in general, equation (\ref{1}) does not have a classical strong solution and the best expectation is to obtain mild solutions in the sense of the above definition (which turn out to be also distributional solutions to (\ref{1})). Another class of solutions are entropy solutions which is a convenient concept of solutions in statistical physics (see, e.g. \cite{ref13}, Chapter 1). But entropy solutions are too special for our purpose, since time marginal densities of solutions to McKean-Vlasov SDEs are in general not entropy solutions. Therefore, distributional solutions constitute the correct class for stochastic analysis.  
\end{remark}
\indent \indent Finally, for later use, we introduce the operator $(L,D(L))$
\begin{equation}
    L:D(L) \rightarrow {\mathcal{D}}'(\mathbb{R}^{d}), \quad Lu:=\sum\limits_{i,j=1}^{d}D_{i}(a_{ij}(x)D_{j}u), \quad D(L)= L_{\text{loc}}^{1}(\mathbb{R}^{d}). \notag
\end{equation}
$(L,D(L))$ is well-defined because $a_{ij} \in C^{\infty}(\mathbb{R}^{d}), i, j=1,2, \dots d$. Moreover, for $u:[0,T] \times \mathbb{R}^{d} \rightarrow \mathbb{R} $ and $\beta:\mathbb{R}^{d} \times \mathbb{R} \rightarrow \mathbb{R} $, we use the shorthand
\begin{equation}
    L \beta(x,u):=L [\beta(\cdot,u(t,\cdot))](x), \;\;  \text{if} \; \beta(\cdot, u(t,\cdot)) \in D(L).\notag
\end{equation}
Here and in the sequel, $\beta(x,u)$ stands for $\beta(x,u(t,x))$.

\section{Uniqueness of distributional solutions to NFPEs} 
\subsection{Existence of mild solutions to NFPEs}
\indent The existence of mild solutions to NFPE (\ref{1}) was proved in \cite{ref13}, Section 2.1 (see also \cite{ref4}). To check Hypotheses (H1)-(H3) in there, we need to introduce the following conditions.\vspace{0.2em} \\
(\textit{\textbf{\romannumeral1}}') $\beta \in C^{1}(\mathbb{R}^{d}\times \mathbb{R})$, $\beta_{r}(x,r)> \delta,\forall x \in \mathbb{R}^{d},\forall r \in \mathbb{R} $ and some $\delta>0$; $\beta (x,0)\equiv 0, \; \beta(x,r)=-\beta(x,-r), \forall x \in \mathbb{R}^{d}, \forall r \in \mathbb{R} $, and for the function $\Phi (x,r):=\displaystyle \frac{\beta(x,r)}{r}, \forall x \in \mathbb{R}^{d}, \forall r \in \mathbb{R}$, where we set $\Phi(x,0):=\displaystyle \lim_{\varepsilon \rightarrow 0}\frac{\beta(x,\varepsilon)}{\varepsilon}=\beta_{r}(x,0), x\in \mathbb{R}^{d} $, we have $\Phi \in C^{2}(\mathbb{R}^{d}\times \mathbb{R}) \cap C_{b}(\mathbb{R}^{d}\times \mathbb{R})$, $\Phi_{x} \in C_{b}(\mathbb{R}^{d}\times \mathbb{R};\mathbb{R}^{d})$,\vspace{0.25em}\\
\noindent (\textit{\textbf{\romannumeral3}}') $b \in C_{b}(\mathbb{R}^{d+1};\mathbb{R}^{d}) \cap C^{1}(\mathbb{R}^{d+1};\mathbb{R}^{d})$. 
\begin{theorem}\label{thm3.1}\cite[Thm.~2.1]{ref13}
Assume that (\textbf{\romannumeral1}'), (\textbf{\romannumeral2}) and (\textbf{\romannumeral3}') hold. Then there exists a continuous contraction semigroup $S(t): [0, \infty ) \rightarrow L^{1}$ such that, for each $u_{0} \in  L^{1}(\mathbb{R}^{d})$, $u=u(t, u_{0})=S(t)u_{0}$ is a mild solution to (\ref{1}) and 
\begin{equation}
    \left | u(t,u_{0})- u(t,\tilde{u}_{0})\right | _{1} \leqslant \left | u_{0}- \tilde{u}_{0} \right | _{1} , \quad \forall u_{0}, \tilde{u}_{0} \in L^{1}, \; t\geqslant 0\notag 
\end{equation}
\begin{equation}
    u \geqslant 0, \text{a.e.} \; \text{on}\;  (0,\infty) \times \mathbb{R}^{d} ,\quad  \text{if} \; u_{0} \geqslant 0,\text{a.e.} \; \text{on} \; \mathbb{R}^{d},   \notag 
\end{equation}
\begin{equation}
    \int_{\mathbb{R}^{d} } u(t,x)dx,=\int_{\mathbb{R}^{d} } u_{0}(x)dx, \quad \forall u_{0} \in L^{1}, \; t \geqslant 0.\notag 
\end{equation}
Moreover, $u$ is a distributional solution to (\ref{1}), that is, for all $\varphi \in C^{\infty}_{0}([0,T)\times \mathbb{R}^{d})$,
\begin{equation}
\begin{split}
   \int^{T}_{0} &\int_{\mathbb{R}^{d}} \bigg(u(t,x)\frac{\partial \varphi }{\partial t}(t,x)+\sum\limits_{i,j=1}^{d}a_{ij}(x)\beta(x,u(t,x))D^{2}_{ij}\varphi(t,x)\\+& u(t,x)b(x, u(t,x)) \cdot \nabla \varphi (t,x) \bigg) dxdt +\int_{\mathbb{R}^{d}}u_{0}(x)\varphi(0,x)dx=0 .  
\end{split}\notag
\end{equation}
\end{theorem}
\begin{proof}
   \rm Set $\displaystyle \tilde a_{ij}(x,u)=\frac{a_{ij}(x)\beta(x,u)}{u}=a_{ij}(x)\Phi(x,u)$. Then, by (\textit{\textbf{\romannumeral1}}') and (\textit{\textbf{\romannumeral2}}), one readily checks that $a_{ij}$ satisfies the following condition
   
\vspace{0.2cm}
\noindent (H1) $\tilde a_{ij} \in C^{2}(\mathbb{R}^{d}\times \mathbb{R}) \cap C_{b}(\mathbb{R}^{d}\times \mathbb{R})$, $(\tilde a_{ij})_{x} \in C_{b}(\mathbb{R}^{d}\times \mathbb{R};\mathbb{R}^{d})$, $\tilde a_{ij}=\tilde a_{ji}$, $\tilde a_{ij}(x,u)=\tilde a_{ij}(x,\left|u\right|), x \in \mathbb{R}^{d}, u \in \mathbb{R}, i,j=1,2,\dots, d$.   
\vspace{0.2cm}

Since $ \tilde a_{ij}(x,u)+(\tilde a_{ij})_{u}(x,u)u=a_{ij}(x)(\Phi(x,u)+u\Phi_{u}(x,u))=a_{ij}(x)\beta_{r}(x,u)$, condition (\textit{\textbf{\romannumeral2}}) implies that $\tilde a_{ij}$ also satisfies\\
(H2) $\sum\limits_{i,j=1}^{d}\big(\tilde a_{ij}(x,u)+(\tilde a_{ij})_{u}(x,u)u\big)\xi_{i}\xi_{j} \geqslant \gamma \delta \left | \xi \right |^{2}, \quad \forall \xi \in \mathbb{R}^{d}, x \in \mathbb{R}^{d}, u \in \mathbb{R}$, for $\gamma$ as in (\textit{\textbf{\romannumeral2}}) and $\delta$ as in (\textit{\textbf{\romannumeral1}}').  \\    
By (\textit{\textbf{\romannumeral3}}'), $b_{i}$ satisfies

\vspace{0.2cm}
\noindent (H3) $b_{i} \in C_{b}(\mathbb{R}^{d} \times \mathbb{R}) \cap C^{1}(\mathbb{R}^{d} \times \mathbb{R}), i=1,2 \dots,d$.\vspace{0.2em}\\
Hence, by Theorem 2.1 in \cite{ref13}, the assertions follow.  $\hfill\square$
\end{proof}

\indent However, the results in \cite{ref13} do not include the additional property that a mild solution starting from data in $L^{1}\cap L^{\infty}$ belongs to $L^{\infty}((0,T)\times\mathbb{R}^{d})$. The following theorem whose proof is given in the Appendix,  provides estimate (\ref{5'''}) below for the corresponding general diffusion terms $a_{ij}(x,u)$ as in \cite{ref13}. It applies, in particular, to the present setting and is needed for the nonlinear Markov property proved in Section 5. 
\begin{theorem}\label{thm3.2}
 Assume that (\textbf{\romannumeral1}'), (\textbf{\romannumeral2}) and (\textbf{\romannumeral3}') hold and, in addition,
\begin{equation}\label{5'}
   \delta _{1}(r):= \sup \left\{  \left | b^{i}_{x} (x,r)\right |; \; i=1, \dots, d,\;x \in \mathbb{R}^{d}\right\}, 
\end{equation}
\begin{equation}\label{5''}
   \delta _{2}(r):= \sup \left\{  \left | (\tilde a_{ij})_{x_{i}x_{j}} (x,r)\right |; \; i,j=1, \dots, d,\;x \in \mathbb{R}^{d}\right\},
\end{equation}
where $\displaystyle \tilde a_{ij}(x,u)=\frac{a_{ij}(x)\beta(x,u)}{u}=a_{ij}(x)\Phi(x,u)$.\\
(\textbf{\romannumeral5}) $\delta_{1}$, $\delta_{2}$ $\in C_{b}(\mathbb{R})$.\\
Then, for each $u_{0} \in L^{1} \cap L^{\infty}$, a mild solution to (\ref{1}) given in Theorem \ref{thm3.1} also satisfies, for some constant $C$ 
\begin{equation}\label{5'''}
 \left|u(t)\right|_{\infty} \leqslant \text{exp}(Ct) \left|u_{0}\right|_{\infty}, \quad \forall t \in [0,T]  .
\end{equation}
\end{theorem}
\subsection{Uniqueness of distributional solutions to NFPEs}
\begin{theorem}\label{thm3.3}
 Assume that \textbf{Hypothesis (\MakeUppercase{\romannumeral 1})} holds. Let $u_{1}, u_{2} \in L^{1}((0,T)\times \mathbb{R}^{d})\cap L^{\infty}((0,T)\times \mathbb{R}^{d})$ be two distributional solutions to (\ref{1}) in $(0,T)\times \mathbb{R}^{d}$ such that $u_{1}-u_{2} \in L^{\infty}((0,T);H^{-1})$ and
 \begin{equation}\label{4}
     \displaystyle\lim_{t \rightarrow 0} \displaystyle \esssup_{s \in (0,t)}\left | \left ( u_{1}(s)-u_{2}(s),\varphi  \right )_{2}\right |=0, \quad  \forall \varphi \in C^{\infty}_{0}(\mathbb{R}^{d}). 
 \end{equation}
 Then, $u_{1}\equiv u_{2} $.
\end{theorem}
\begin{proof}
\rm We adapt the technique developed in \cite{ref11}, though it has to be modified to apply to our substantially more general situation. We first note that for $i=1,2$, $\beta(x,u_{i})\in L^{1}_{\text{loc}}=D(L)$. Hence
\begin{equation}
\begin{split}
 \sum_{i,j=1}^{d}D^{2}_{ij}(a_{ij}(x)\beta(x,u_{i}))
&=\sum_{i,j=1}^{d}D_i(a_{ij}(x)D_j\beta(x,u_{i}))+\operatorname{div}(A(x)\beta(x,u_{i}))\\& =L\beta(x,u_{i})+\operatorname{div}(A(x)\beta(x,u_{i})),  \notag
\end{split}
\end{equation}
where $A(x)=\{A_i(x)\}_{i=1}^{d}$ and $A_i(x)=\displaystyle \sum_{j=1}^{d}D_j a_{ij}(x)$. Below we also set $b^{*}(x,u):=b(x,u)u$, and define $N := \text{max}\left \{ \left | u_{1} \right |_{\infty}, \left | u_{2} \right |_{\infty}   \right \} $. \\
We set 
\begin{equation}
    z:=u_{1}-u_{2},\ w:=\beta(x,u_{1})-\beta(x,u_{2}),\ \zeta:=b^{*}(x, u_{1})-b^{*}(x, u_{2}), \forall x \in \mathbb{R}^{d}.\notag
\end{equation}
Then, by Hypothesis (\MakeUppercase{\romannumeral 1}) on $[0,T] \times \mathbb{R}^{d}$,
\begin{equation}\label{7a}
  \left| \zeta \right| \leqslant \alpha_{2}\left| w \right|  , 
\end{equation}
\begin{equation}\label{7b}
 wz \geqslant \alpha_{1}\left| w \right|^{2}  ,
\end{equation}
where $\alpha_{2}=\alpha_{[-N,N]}$ and $\alpha^{-1}_{1}=\left | \beta_{r}\right |_{L^{\infty}(\mathbb{R}^{d}\times[-N,N])}$. Furthermore, our conditions imply that $z,w, \zeta \in L^{1}((0,T)\times \mathbb{R}^{d})\cap L^{\infty}((0,T)\times \mathbb{R}^{d})$ and 
\begin{equation}\label{2}
    z_{t}- Lw +\text{div}(\zeta-Aw)=0, \quad \text{in} \ {\mathcal{D}}'((0,T)\times \mathbb{R}^{d}). 
\end{equation}
We define
\begin{equation}
    z_{\delta }:=z*\theta _{\delta }, w_{\delta }:=w*\theta _{\delta },\zeta _{\delta }:=\zeta*\theta _{\delta }, (Aw) _{\delta }:=(Aw)*\theta _{\delta }, (Lw) _{\delta}:=(Lw)*\theta _{\delta },\notag
\end{equation}
where $\theta \in C^{\infty}_{0}(\mathbb{R}^{d}) $, $\theta_{\delta}(x) \equiv \frac{1}{\delta ^{d}}\theta (\frac{x}{\delta })$, $\delta >0$, is a standard mollifier. Since $ z_{\delta }, w_{\delta } \in L^{2}(0,T;H^{2})$, and $\zeta_{\delta}$, $(Aw) _{\delta }$, $ (Lw)_{\delta}$, $\text{div}\zeta_{\delta }$, $\text{div}(Aw) _{\delta }$ $\in L^{2}(0,T;L^{2})$, we have
\begin{equation}\label{8*}
    (z_{\delta})_{t}- (Lw)_{\delta } +\text{div}(\zeta_{\delta }-(Aw) _{\delta })=0, \quad \text{in} \ {\mathcal{D}}'((0,T); L^{2}).
\end{equation}
We recall that by our assumptions on $a_{ij}$ and standard elliptic regularity theory (see, e.g. \cite[Theorem 1 and Theorem 2.2]{ref30}) we have that
\begin{equation}
  \epsilon I -L_{| H^{2}}: H^{2} \longrightarrow L^{2}  \notag
\end{equation}
is a homeomorphism for every $\epsilon >0$, where $I$ denotes the identity operator. We define its inverse operator $\Psi_{\varepsilon}  $ as follows: for any $\varepsilon >0$,
\begin{equation}
  \Psi_{\varepsilon}: L^{2} \rightarrow H^{2} , \notag
\end{equation}
\begin{equation}
  \Psi_{\varepsilon}(u)=(\varepsilon I-L_{| H^{2}})^{-1}u, \quad \forall u \in L^{2} . \notag
\end{equation}
Then $\Psi_{\varepsilon}(z_{\delta})$, $ \Psi_{\varepsilon}(w_{\delta})$ $\in L^{2}(0,T;H^{2})$; $\Psi_{\varepsilon}(\text{div}\zeta _{\delta })$, $ \Psi_{\varepsilon}(\text{div}(Aw) _{\delta})$ $\in L^{2}(0,T;L^{2})$. \\
Therefore, applying $\Psi_{\epsilon}$ to (\ref{8*}) yields
\begin{equation}\label{0.6}
\begin{split}
    (\Psi_{\varepsilon}(z_{\delta }))_{t}= &\Psi_{\varepsilon}(Lw_{\delta }) +\Psi_{\varepsilon}(\text{div}((Aw) _{\delta }-\zeta _{\delta }))+\Psi_{\varepsilon}(R_{\delta})\\=&L\Psi_{\varepsilon}(w_{\delta}) +\Psi_{\varepsilon}(\text{div}(Aw)_{\delta})-\Psi_{\varepsilon}(\text{div}\zeta _{\delta })+\Psi_{\varepsilon}(R_{\delta}), \quad \text{in} \ {\mathcal{D}}'((0,T); L^{2}),
\end{split}    
\end{equation}
where $R_{\delta}=(Lw)_{\delta }-Lw_{\delta}$. We note that, since $w(t), w_{\delta}(t) \in L^{2}$, we have
\begin{equation}
    (Lw(t))_{\delta} \longrightarrow Lw(t) \quad \text{in}\;\; H^{-2},\notag
\end{equation}
\begin{equation}
    Lw_{\delta}(t) \longrightarrow Lw(t) \quad \text{in}\;\; H^{-2}.\notag
\end{equation}
Since $R_{\delta}=(Lw)_{\delta }-Lw_{\delta}=(I-L)w_{\delta}-((I-L)w)_{\delta}$, for some constant $C>0$
\begin{equation}
\left| R_{\delta}(t)\right|_{-2} \leqslant C\left| w(t)\right|_{2} \quad \text{a.e.}\;\; t \in [0,T].\notag
\end{equation}
Since $w \in L^{2}(0,T;L^{2})$, we altogether obtain
\begin{equation}\label{10'}
    \lim_{\delta \rightarrow 0} \left| R_{\delta} \right|_{L^{2}(0,T;H^{-2})}=0. 
\end{equation}
We also note that
\begin{equation}\label{3}
    z_{\delta }(t)=\varepsilon\Psi_{\varepsilon}(z_{\delta }(t))-L\Psi_{\varepsilon}(z_{\delta }(t)).
\end{equation}
We set 
\begin{equation}
    h_{\varepsilon, \delta}(t)=(\Psi_{\varepsilon}(z_{\delta }(t)),z_{\delta }(t))_{2}, \quad \text{a.e.}\;t \in (0,T).\notag
\end{equation}
Below we shall prove two claims.
\begin{claim}\label{clm1}
    $ \displaystyle\lim_{t \rightarrow 0} \displaystyle \esssup_{s \in (0,t)}h_{\varepsilon,\delta }(s)=0 $.
\end{claim}
\begin{proof}
\rm By (\ref{0.6}) we have $ \frac{d}{dt}\Psi_{\varepsilon}(z_{\delta}) \in L^{2}(0,T;L^{2}) $, where $\frac{d}{dt}$ is taken in the sense of vector-valued distributions in $L^{2}$ on $(0,T)$. Thus, $z_{\delta}$, $\Psi_{\varepsilon}(z_{\delta}) \in H^{1}(0,T;L^{2})$. On the other hand, we have $\Psi_{\varepsilon}(z_{\delta}) \in L^{2}(0,T;H^{2})$.\\
Therefore, 
\begin{equation}
    \Psi_{\varepsilon}(z_{\delta}) \in C([0,T];H^{1}),\notag
\end{equation}
since the spaces $L^{2}$ and $H^{2}$ are in duality with the pivot space $H^{1}=(H^{2},L^{2})_{\frac{1}{2}}$, after modifying $t \rightarrow  \Psi_{\varepsilon}(z_{\delta}(t))$ on a subset of measure zero (see, e.g. \cite{ref2} and \cite{ref11}). Since $t \rightarrow  \Psi_{\varepsilon}(z_{\delta}(t))$ has an $H^{1}$-continuous version on $[0,T]$, there exists $f\in H^{1}$ such that 
\begin{equation}\label{5}
    \displaystyle\lim_{t \rightarrow  0} \Psi_{\varepsilon}(z_{\delta}(t))=f \quad \text{in} \ H^{1}. 
\end{equation}
By (\textbf{\romannumeral2}) and \eqref{3}, we observe that
\begin{equation}\label{8}
\begin{split}
   h_{\varepsilon, \delta}(t)=&(\Psi_{\varepsilon}(z_{\delta }(t)),z_{\delta }(t))_{2}=\int_{\mathbb{R}^{d}}\Psi_{\varepsilon}(z_{\delta}(t))(\varepsilon \Psi_{\varepsilon}(z_{\delta }(t)-L\Psi_{\varepsilon}(z_{\delta }(t))dx\\=&\varepsilon \left |\Psi_{\varepsilon}(z_{\delta }(t) \right |_{2}^{2}-\int_{\mathbb{R}^{d}}\Psi_{\varepsilon}(z_{\delta}(t)L\Psi_{\varepsilon}(z_{\delta }(t)dx\\=& \varepsilon \left |\Psi_{\varepsilon}(z_{\delta }(t) \right |_{2}^{2}+\int_{\mathbb{R}^{d}}\sum\limits_{i,j=1}^{d}a_{ij}(x)D_{j}\Psi_{\varepsilon}(z_{\delta }(t)D_{i}\Psi_{\varepsilon}(z_{\delta}(t)dx \\ \geqslant&\varepsilon \left |\Psi_{\varepsilon}(z_{\delta }(t) \right |_{2}^{2}+\gamma \left | \nabla\Psi_{\varepsilon}(z_{\delta}(t) \right |_{2}^{2} \geqslant 0.
\end{split} 
\end{equation}
Furthermore,
\begin{equation}
\begin{split}
0 \leqslant  h_{\varepsilon, \delta}(s)=(&\Psi_{\varepsilon}(z_{\delta }(s)),z_{\delta }(s))_{2} \leqslant \left |\Psi_{\varepsilon}(z_{\delta }(s)-f \right |_{H^{1}}\cdot \left |z_{\delta }(s)\right |_{H^{-1}}\\+ &\left |f -\varphi \right |_{H^{1}}\cdot \left |z_{\delta }(s)\right |_{H^{-1}}+\left |(\varphi *\theta _{\delta},z(s)) _{2}\right |,\quad \forall \varphi \in C^{\infty}_{0}(\mathbb{R}^{d}),\ s \in (0,T). \notag
\end{split}     
\end{equation} 
 Hence, by \eqref{4}, \eqref{5},
 \begin{equation}
 \begin{split}
 0 \leqslant& \displaystyle\lim_{t \rightarrow 0} \displaystyle \esssup_{s \in (0,t)}h_{\varepsilon,\delta }(s)
 \\ \leqslant & \left(\displaystyle\lim_{t \rightarrow 0}\left | \Psi_{\varepsilon}(z_{\delta }(t)-f\right |_{H^{1}}+\left |f -\varphi \right |_{H^{1}}\right )\left |z_{\delta }\right |_{L^{\infty}(0,T;H^{-1})}+\displaystyle\lim_{t \rightarrow 0} \displaystyle \esssup_{s \in (0,t)}\left | \left ( \varphi *\theta _{\delta },z(s)  \right )_{2}\right |\\=&\left |f -\varphi \right |_{H^{1}}\left |z_{\delta }\right |_{L^{\infty}(0,T;H^{-1})}.
 \end{split}\notag
 \end{equation}
Therefore, Claim \ref{clm1} follows since $ C^{\infty}_{0}(\mathbb{R}^{d})$ is dense in $ H^{1}(\mathbb{R}^{d})$.
$\hfill\square $   
\end{proof}

\begin{claim}\label{clm2}
    Each $h_{\varepsilon, \delta}$ has an absolutely continuous version and for these versions, there exists a choice of the mollification parameter $\delta=\delta(\varepsilon) \rightarrow 0$ such that $\displaystyle\lim_{\varepsilon \rightarrow 0}h_{\varepsilon, \delta(\varepsilon)}(t)=0 $ uniformly on $t \in [0,T]$, i.e. $\displaystyle\lim_{\varepsilon \rightarrow 0}h_{\varepsilon, \delta(\varepsilon)}=0 $ in $C([0,T])$.
\end{claim}
\begin{proof}
\rm By (\ref{8*}) we have $\displaystyle \frac{d}{dt}z_{\delta} \in L^{2}(0,T;L^{2})$, hence $ z_{\delta} \in H^{1}(0,T;L^{2})$ and thus $[0,T] \ni t \mapsto z_{\delta}(t) \in L^{2}$ and, as seen below, also $[0,T] \ni t \mapsto \Psi_{\varepsilon}(z_{\delta}(t)) \in L^{2}$ are absolutely continuous (see \cite{ref2}, p. 23, Theorem 1.17). Then, by the definition of $ h_{\varepsilon, \delta}(t)$, \eqref{3} and \eqref{0.6}, we have 
\begin{equation}\label{13'}
\begin{split}
h'_{\varepsilon,\delta }(t)=& 2((\Psi_{\varepsilon}(z_{\delta }(t)))_{t},z_{\delta }(t))_{2}\\=&2\left(L\Psi_{\varepsilon}(w_{\delta }(t)) +\Psi_{\varepsilon}(\text{div}((Aw) _{\delta}(t)))-\Psi_{\varepsilon}(\text{div}(\zeta _{\delta}(t)),z_{\delta}(t)\right)_{2}+2(R_{\delta}(t),\Psi_{\varepsilon}(z_{\delta}(t)))_{2}\\=&2\varepsilon(\Psi_{\varepsilon}(z_{\delta}(t)),w_{\delta }(t))_{2}-2(w_{\delta }(t),z_{\delta }(t))_{2}+2(R_{\delta}(t),\Psi_{\varepsilon}(z_{\delta }(t)))_{2}\\&-2(\nabla \Psi_{\varepsilon}(z_{\delta }(t)),(Aw) _{\delta }(t))_{2}+2(\nabla \Psi_{\varepsilon}(z_{\delta }(t)),\zeta _{\delta }(t))_{2}, \quad \text{a.e.}\  t\in (0,T). 
\end{split}
\end{equation}
Next, we derive several estimates for $h'_{\varepsilon, \delta}(t)$.\\
a) First by (\ref{7b}) we have \begin{equation}
   (w_{\delta }(t),z_{\delta }(t))_{2} \geqslant \alpha_{1}\left |w(t) \right |_{2}^{2}+\gamma_{\delta}(t),\notag 
\end{equation}
where $ \gamma_{\delta}(t) =(w_{\delta }(t),z_{\delta }(t))_{2}-(w(t),z(t))_{2}$.\\
b) By (\ref{7a}) we have \begin{equation}
\begin{split}
   (\nabla \Psi_{\varepsilon}(z_{\delta }(t)),\zeta _{\delta}(t))_{2}-&  (\nabla \Psi_{\varepsilon}(z_{\delta}(t)),(Aw) _{\delta }(t))_{2} \leqslant \left | \nabla\Psi_{\varepsilon}(z_{\delta }(t) \right |_{2} \cdot (\left | \zeta (t) \right |_{2}+\left | (Aw)(t) \right |_{2}) \\ \leqslant& \left | \nabla\Psi_{\varepsilon}(z_{\delta }(t) \right |_{2} \cdot (\alpha_{2}\left | w(t)\right |_{2}+\left | A \right |_{\infty}\left | w(t)\right |_{2}) \\ \leqslant& C \left | \nabla\Psi_{\varepsilon}(z_{\delta }(t) \right |_{2} \left | w(t)\right |_{2},\notag
\end{split}   
\end{equation}
where $C=\alpha_{2}+\left | A \right |_{\infty}$.\\
c) For all $\varepsilon, \delta \in (0,1)$, we have \begin{equation}\label{7}
    \varepsilon (\left|\Psi_{\varepsilon}(z_{\delta }(t))\right|,\left|w_{\delta }(t)\right|)_{2} \leqslant \varepsilon \left | \Psi_{\varepsilon}(z_{\delta}(t)) \right |_{\infty} \left | w_{\delta}(t)  \right |_{1} \leqslant c \left | z(t) \right |_{\infty} \left | w(t)  \right |_{1}, \quad \text{a.e.} \ t \in (0,T).
\end{equation}
\begin{equation}\label{6}
 \displaystyle\lim_{\varepsilon \rightarrow 0}\sup_{\delta \in (0,1)} \left| \varepsilon(\Psi_{\varepsilon}(z_{\delta}(s))\right|_{\infty}=0, \quad \text{a.e.} \ s \in (0,T)
\end{equation}
To prove the estimates (\ref{7}) and (\ref{6}), let $K_{\varepsilon}$
be the integral kernel of $(\varepsilon I-L)^{-1}$.
By Theorem 7 in \cite{ref1}, there exists $c \in (1, \infty)$ such that
$c^{-1}K_{\varepsilon}^{\Delta}(x-\xi)\leqslant K_{\varepsilon}(x,\xi) \leqslant cK_{\varepsilon}^{\Delta}(x-\xi)$, $\forall x, \xi \in \mathbb{R}^{d}$, where $K_{\varepsilon}^{\Delta}$ is the integral kernel of $(\varepsilon I-\Delta)^{-1}$.\\
Therefore, for a.e. $t \in [0,T]$, we have 
\begin{equation}
\begin{split}
    \left | \varepsilon \Psi_{\varepsilon}(z_{\delta }(t))(x) \right |\leqslant & \varepsilon \int_{\mathbb{R}^{d}}  K_{\varepsilon}(x,\xi)\left|z_{\delta}(t)(\xi)\right|d \xi  \\ \leqslant& \varepsilon \int_{\mathbb{R}^{d}} c K_{\varepsilon}^{\Delta}(x-\xi)\left|z_{\delta}(t)(\xi)\right|d \xi =c \varepsilon (\varepsilon I-\Delta)^{-1}\left|z_{\delta }(t)\right|(x) ,
\end{split}    \notag
\end{equation}
It is well known that
\begin{equation}\label{15'}
  \varepsilon \left | (\varepsilon I-\Delta)^{-1}y\right|_{p} \leqslant \left | y\right|_{p},\quad \forall y \in L^{p}, \; p \in [1, \infty]   ,
\end{equation}
Hence, $\varepsilon \left| \Psi_{\varepsilon}(z_{\delta}(t))\right|_{\infty} \leqslant c\left|z_{\delta}(t) \right|_{\infty}$, and therefore (\ref{7}) holds. \\
To prove (\ref{6}) we proceed similarly as in \cite[Proof of Thorem 3.1]{ref12} (see also \cite{ref17}) by writing for a.e. $t \in (0,T)$,
\begin{equation}
    \varepsilon(\varepsilon I-\Delta)^{-1}\left|z_{\delta}(t)\right|(x)=\varepsilon ^{\frac{d}{2}}\int_{\mathbb{R}^{d}} K_{1}^{\Delta} (\sqrt{\varepsilon} (x-\xi)) \left|z_{\delta }(t)(\xi)\right|d \xi ,\quad x \in \mathbb{R}^{d}, \notag
\end{equation}
where $K_{1}^{\Delta} $ is the kernel associated with the operator $ (I-\Delta)^{-1}$ (see \cite{ref24}), i.e.
\begin{equation}
   (I-\Delta)^{-1}(y)(x)=\int_{\mathbb{R}^{d}} K_{1}^{\Delta}(x-\xi)y(\xi)d \xi, \quad x \in \mathbb{R}^{d} .\notag
\end{equation}
This yields for a.e. $ x \in \mathbb{R}^{d} $,
\begin{equation}
  \left |  \varepsilon(\varepsilon I-\Delta)^{-1}\left|z_{\delta }(t)\right|(x) \right | \leqslant c_{r}\varepsilon ^{\frac{d}{2}} \left|z(t) \right|_{1} + \varepsilon ^{\frac{d}{2}} \left|z(t) \right|_{\infty} \int_{\sqrt{\varepsilon} \left|x-\xi\right| \leqslant r} K_{1}^{\Delta} (\sqrt{\varepsilon} (x-\xi))d \xi, \quad r>0,\notag
\end{equation}
where $c_{r}=\text{sup}\left \{ K(x); \left| x\right| \geqslant r \right \} < \infty $, $\forall r >0$. Then, for a.e. $t \in (0,T)$, we get
\begin{equation}
    \displaystyle\limsup_{\varepsilon \downarrow  0}\sup_{\delta \in (0,1)}\left | \varepsilon (\varepsilon I-\Delta)^{-1}\left|z_{\delta}(t)\right|\right |_{\infty} \leqslant \left| z(t)\right|_{\infty}\int_{ \left|\xi\right| \leqslant r} K_{1}^{\Delta} (\xi)d \xi, \quad \forall r>0,\notag
\end{equation}
letting $ r \rightarrow 0$, since $K_{1}^{\Delta} $ as a probability kernel is locally integrable, (\ref{6}) follows.\\
By (\ref{7}), (\ref{6}) and the dominated convergence theorem, it follows that
\begin{equation}\label{dc}
 \displaystyle\lim_{\varepsilon \rightarrow 0} \varepsilon\int^{T}_{0}\sup_{\delta \in (0,1)} (\left|\Psi_{\varepsilon}(z_{\delta }(s))\right|,\left|w_{\delta}(s)\right|)_{2}ds =0. 
\end{equation}
d) For every $\varepsilon \in (0,1)$, there exists $\delta(\varepsilon) \in (0,\varepsilon]$ such that 
\begin{equation} \label{d}
    \displaystyle \int_{0}^{T}\left|(R_{\delta(\varepsilon) }(s),\Psi_{\varepsilon}(z_{\delta(\varepsilon) }(s)))_{2}\right|ds \leqslant \varepsilon \left| z \right|_{L^{2}(0,T;L^{2})}. 
\end{equation}
It easily follows by (\ref{15'}) that the operator norm of $\varepsilon \Psi_{\varepsilon}: L^{2} \longrightarrow H^{2} $ is less than one. Hence for a.e. $s \in [0,T]$
\begin{equation}
\begin{split}
 \left|(R_{\delta }(s),\Psi_{\varepsilon}(z_{\delta }(s)))_{2}\right|  \leqslant \left|R_{\delta }(s) \right|_{-2}\cdot \left|\Psi_{\varepsilon}(z_{\delta }(s)) \right|_{H^{2}}\leqslant  (1/\varepsilon) \left|R_{\delta}(s) \right|_{-2}\cdot \left|z(s) \right|_{2}   .\notag
\end{split} 
\end{equation}
By (\ref{10'}) we can choose $\delta(\varepsilon) \in (0, \varepsilon]$ such that $\left| R_{\delta(\varepsilon)} \right|_{L^{2}(0,T;H^{-2})} \leqslant \varepsilon^{2}$. Hence, (\ref{d}) holds. Then by Claim \ref{clm1}, (\ref{8}), (\ref{13'}), a), b), d), we have 
\begin{equation}
\begin{split}
    0 \leqslant h_{\varepsilon,\delta(\varepsilon)}(t) \leqslant&  2\varepsilon \int^{t}_{0}  (\left|\Psi_{\varepsilon}(z_{\delta(\varepsilon) }(s))\right|,\left|w_{\delta(\varepsilon) }(s)\right|)_{2} ds-2\alpha_{1}\int_{0}^{t} \left| w(s)\right|_{2}^{2}ds+2\int^{t}_{0}\left | \gamma_{\delta(\varepsilon)}(s)\right | ds\\+& 2 \varepsilon \left| z \right|_{L^{2}(0,T;L^{2})}+\frac{C^{2}}{\alpha_{1}}\int^{t}_{0}\left | \nabla\Psi_{\varepsilon}(z_{\delta(\varepsilon)}(s) \right |^{2}_{2} ds+\alpha_{1}\int^{t}_{0}\left |w(s) \right |_{2}^{2}ds \\ \leqslant& \tilde \zeta_{\varepsilon}+\frac{C^{2}}{\alpha_{1}\gamma} \int^{t}_{0}h_{\varepsilon,\delta(\varepsilon)}(s)ds
    \end{split} 
\end{equation}
where we used (\ref{8}) again in the last step and set
\begin{equation}
    \tilde \zeta_{\varepsilon}:=2\varepsilon \int^{T}_{0}  (\left|\Psi_{\varepsilon}(z_{\delta(\varepsilon)}(s))\right|,\left|w_{\delta(\varepsilon)}(s)\right|)_{2}ds+2 \varepsilon \left| z \right|_{L^{2}(0,T;L^{2})}+2\int^{T}_{0}\left | \gamma_{\delta(\varepsilon)}(s)\right | ds.\notag
\end{equation}
Clearly by (\ref{dc}) we have $\displaystyle\lim_{\varepsilon \rightarrow 0}\tilde \zeta_{\varepsilon}=0 $. Hence by Gronwall's inequality, we obtain
\begin{equation}
     0 \leqslant h_{\varepsilon, \delta(\varepsilon)}(t) \leqslant \tilde \zeta_{\varepsilon}\text{exp}(\frac{C^{2}}{\alpha_{1}\gamma}t), \quad \forall t \in [0,T].
\end{equation}
This implies that $h_{\varepsilon, \delta(\varepsilon)}(t)\rightarrow 0$ as $\varepsilon \rightarrow 0$ uniformly on $[0,T]$ and hence Claim \ref{clm2} is proved. $\hfill\square $
\end{proof}
\textbf{Proof of Theorem \ref{thm3.3} (continued):} We observe from (\ref{8}) that, $h_{\varepsilon, \delta(\varepsilon)}(t)$ can be rewritten as $h_{\varepsilon, \delta(\varepsilon)}(t)=\varepsilon \left | \Psi_{\varepsilon}(z_{\delta(\varepsilon) }(t))\right |^{2}_{2} + \left | (-L)^{\frac{1}{2}}\Psi_{\varepsilon}(z_{\delta(\varepsilon) }(t))\right |^{2}_{2}$ (see, e.g. Theorem 1.3.1 in \cite{ref31} or p. 27 in \cite{ref20}).
Hence, by (\ref{6}) and Claim \ref{clm2}, the right-hand side of (\ref{3}) converges to $0$ in ${\mathcal{D}}'$. Therefore, $0=\displaystyle\lim_{\varepsilon \rightarrow 0}z_{\delta(\varepsilon)}(t)=z(t)$ in ${\mathcal{D}}'$ for a.e. $t \in (0,T)$, which implies $u_{1}\equiv u_{2} $.

$\hfill\square $    
\end{proof}

\begin{corollary}\label{col1}
Assume that Hypothesis (\MakeUppercase{\romannumeral 1}) holds. Let $T>0$ and $u_{0} \in \mathcal{P}\cap L^{\infty} $, and let $u_{1}, u_{2} \in L^{1}((0,T)\times \mathbb{R}^{d})\cap L^{\infty}((0,T)\times \mathbb{R}^{d})$ be two nonnegative distributional solutions to (\ref{0}). Then $u_{1} \equiv u_{2}$.   
\end{corollary}
\begin{proof}
\rm Because of Hypothesis (\MakeUppercase{\romannumeral 1}) (\textbf{\romannumeral1})-(\textbf{\romannumeral3}), it is elementary to check that by (\ref{0})
\begin{equation}\label{18'}
    \int_{\mathbb{R}^{d}}u_{i}(t,x)dx= \int_{\mathbb{R}^{d}}u_{0}(x)dx \;\; \text{for a.e.}\;t \in (0,T),\; i=1,2.  
\end{equation}
Hence, it follows by Lemma 2.3 in \cite{ref22}, there exist $dt \otimes dx$-versions $\bar{u}_{1}$ of $u_{1}$ and $\bar{u}_{2}$ of $u_{2}$ such that for $\bar{u}_{1}(t)(dx):=u_{1}(t,x)dx $, $\bar{u}_{2}(t)(dx):=u_{2}(t,x)dx $, $\bar{u}_{1}(0,x)dx=\bar{u}_{2}(0,x)dx=u_{0}(x)dx $ , the map $[0,T] \ni t \mapsto \bar{u}_{i}(t,x)dx$ is narrowly continuous for $i=1,2$. In particular, (\ref{18'}) holds for every $t\in [0,T]$, and hence $u_{1}, u_{2} \in L^{\infty}((0,T);L^{1} \cap L^{\infty}) \subset L^{\infty}(0,T;L^{2}) \subset L^{\infty}(0,T;H^{-1})  $. Therefore, for every $\varphi \in C^{\infty}_{0}(\mathbb{R}^{d})$,
\begin{equation}
     \displaystyle\lim_{t \rightarrow 0} \displaystyle \esssup_{s \in (0,t)}\left | \left ( u_{1}(s)-u_{2}(s),\varphi  \right )_{2}\right |= \displaystyle\lim_{t \rightarrow 0} \left | \int_{\mathbb{R}^{d}}\left ( \bar{u_{1}}(t,x)-\bar{u_{2}}(t,x) \right )\varphi (x)dx\right | =0. \notag 
 \end{equation}
So, (\ref{4}) holds and Theorem \ref{thm3.3} implies the assertion.
$\hfill\square $

\end{proof}

\indent By Theorem \ref{thm3.2} and Corollary \ref{col1}, we obtain the following existence and uniqueness result for NFPE (\ref{1}).
\begin{corollary}\label{col3.2}
    Under Hypotheses (\textbf{\romannumeral1}'),(\textbf{\romannumeral2}), (\textbf{\romannumeral3}'), (\textbf{\romannumeral4})-(\textbf{\romannumeral5}), for each $u_{0} \in \mathcal{P}\cap L^{\infty} $, NFPE (\ref{1}) has a unique distributional solution
    \begin{equation}
     u \in L^{1}((0,T); L^{1})\cap L^{\infty}((0,T)\times \mathbb{R}^{d})   ,\notag
    \end{equation}
which is in fact a weakly continuous probability solution to (\ref{1}).    
\end{corollary}

\section{Existence and uniqueness of the corresponding McKean-Vlasov SDEs}
\noindent In this section, we prove weak uniqueness for the McKean-Vlasov SDE (\ref{9}) 
\begin{equation}
\begin{aligned}
   & dX(t)=b(X(t), u(t, X(t)))dt+ \displaystyle \sqrt{\frac{2\beta(X(t),u(t,X(t)))}{u(t,X(t))}}\sigma(X(t))dW(t),\forall t>0, \\
 & \mathcal{L}_{X(0)}(dx)=u_{0}(x)dx,\quad \mathcal{L}_{X(t)}(dx)=u(t,x)dx,\quad t>0, \notag
\end{aligned}  
\end{equation}
corresponding to the nonlinear Fokker-Planck equation (\ref{1}), where $u$ is a distributional solution to it with $u_{0} \in \mathcal{P}\cap L^{\infty} $ under Hypothesis (\MakeUppercase{\romannumeral 1}). By the superposition principle, first applied to the linearized equation (see (\ref{18}) below) and subsequently derived for NFPE (\ref{1}), it follows (see, e.g. \cite{ref4}, \cite{ref13} and \cite{ref25}) that there exists a probabilistically weak solution $X(t)$ to the McKean-Vlasov equation (\ref{9}) with law density $u(t)$. More precisely, there is a probability space $(\Omega ,\mathscr{F},\mathbb{P})$ with normal filtration $ (\mathscr{F}_{t})_{t \geqslant 0}$ and an $(\mathscr{F}_{t})$-Brownian motion $W(t)$ on it with values in $\mathbb{R}^{d}$ and a (with respect to $t$) continuous $(\mathscr{F}_{t})$-progressively measurable map $X: [0, \infty) \times \Omega \rightarrow \mathbb{R}^{d}$, which is a solution to (\ref{9}), and $u(t,x)dx=\mathbb{P}\circ (X(t))^{-1}(dx)$, $u_{0}(x)dx=\mathbb{P}\circ (X(0))^{-1}(dx)$.\\
\indent Moreover, we prove that, by applying the restricted Yamada-Watanabe theorem to McKean-Vlasov SDEs in \cite{ref32}, under Hypotheses (\textbf{\romannumeral1}'),(\textbf{\romannumeral2}), (\textbf{\romannumeral3}'), (\textbf{\romannumeral5}), the probabilistically weak solution $X(t)$ to the McKean-Vlasov SDE (\ref{9}) is also the unique probabilistically strong solution, i.e. it is a measurable adapted functional of the driving Brownian motion $W(t)$ and the initial condition $X(0)$. For the precise definition of a (unique) strong solution to McKean-Vlasov SDEs, we refer the reader to \cite[Def.~2.5 and Def.~2.6]{ref32}.

\subsection{Weak uniqueness of the corresponding McKean-Vlasov SDEs}\label{section4.1}
To prove weak uniqueness of (\ref{9}), we first establish the so-called ``linearized uniqueness'' for the ``frozen'' linearized version of NFPE (\ref{1}), namely
\begin{equation}\label{18}
\begin{split}
& \nu_{t}-\sum\limits_{i,j=1}^{d} D^{2}_{ij}\left(\displaystyle \frac{a_{ij}(x)\beta(x,u)}{u}\nu\right)+ \text{div}(b(x,u)\nu)=0,\quad \text{in} \ {\mathcal{D}}'((0,T)\times \mathbb{R}^{d}), \\& 
 \nu(0,x)=\nu_{0}(x), \quad x\in \mathbb{R}^{d},
\end{split}
\end{equation}
where $u \in L^{1}((0,T)\times \mathbb{R}^{d})\cap L^{\infty}((0,T)\times \mathbb{R}^{d}) $ is a distributional solution to NFPE (\ref{1}), i.e., $ \forall \varphi \in C^{\infty}_{0}([0,T)\times \mathbb{R}^{d})$, $\nu_{0} \in L^{1}(\mathbb{R}^{d})$,
\begin{equation}\label{10}
\begin{split}
   \int^{T}_{0} &\int_{\mathbb{R}^{d}} \bigg(\frac{\partial \varphi }{\partial t}(t,x)+\sum\limits_{i,j=1}^{d}a_{ij}(x)\frac{\beta(x,u(t,x))}{u(t,x)}D^{2}_{ij}\varphi(t,x)\\+& b(x, u(t,x)) \cdot \nabla \varphi (t,x) \bigg)\nu(t,x) dxdt +\int_{\mathbb{R}^{d}}\nu_{0}(x)\varphi(0,x)dx=0.  
\end{split}
\end{equation}
\begin{theorem}\label{thm4.1} (Linearized Uniqueness)
Assume that Hypothesis (\MakeUppercase{\romannumeral 1}) holds. Let $T>0$, let $u \in L^{1}((0,T)\times \mathbb{R}^{d})\cap L^{\infty}((0,T)\times \mathbb{R}^{d}) $ be a distributional solution to NFPE (\ref{1}), and let $\nu_{1}, \nu_{2} \in L^{1}((0,T)\times \mathbb{R}^{d})\cap L^{\infty}((0,T)\times \mathbb{R}^{d})$ be two distributional solutions to (\ref{10}) for $\nu_{0} \in \mathcal{P}\cap L^{\infty}$ such that $\nu_{1}-\nu_{2} \in L^{\infty}((0,T);H^{-1})$. If (\ref{4}) holds with $\nu_{i}$ replacing $u_{i}$, $i=1,2$, then $\nu_{1} \equiv \nu_{2}$.
\end{theorem}
\begin{proof}
 \rm Since the proof is almost the same as that for Theorem \ref{thm3.1}, we only give a sketch of it below.\\
 We first rewrite the original linearized equation using the operator $L $ as follows:
 \begin{equation}
 \begin{split}
 \displaystyle & \frac{\partial \nu}{\partial t}-L(\frac{\beta(x,u)}{u}\nu) + \text{div}((b(x,u)-A\frac{\beta(x,u)}{u})\nu)=0 \quad \text{in} \ {\mathcal{D}}'((0,T)\times \mathbb{R}^{d}), \\& 
 \nu(0,x)=\nu_{0}(x), \quad x\in \mathbb{R}^{d}.\notag
\end{split}
 \end{equation}
We set 
\begin{equation}
    z:=\nu_{1}-\nu_{2},\ w:=\frac{\beta(x,u)}{u}(\nu_{1}-\nu_{2}), \notag
\end{equation}
and define a mollification analogous to that in the proof of Theorem \ref{thm3.1}
\begin{equation}
    z_{\delta }:=z*\theta _{\delta }, w_{\delta }:=w*\theta _{\delta},\zeta _{\delta }:=(b(x,u)(\nu_{1}-\nu_{2}))*\theta _{\delta}, (Aw) _{\delta}:=(Aw)*\theta _{\delta}. \notag
\end{equation}
We note that our assumptions imply
\begin{equation}
    \frac{\beta(x,u)}{u}, b(x,u) \in L^{\infty}((0,T)\times \mathbb{R}^{d}), \quad \frac{\beta(x,u)}{u} \geqslant 0, \forall x \in \mathbb{R}^{d}, \notag
\end{equation}
we can then define the same $\Psi_{\varepsilon} $ and $ h_{\varepsilon, \delta}(t)$ as in the proof of Theorem \ref{thm3.1}. Repeating the proof of Claim \ref{clm1}, we obtain $ \displaystyle\lim_{t \rightarrow 0} \displaystyle \esssup_{s \in (0,t)}h_{\varepsilon,\delta}(s)=0 $. Repeating the proof of Claim \ref{clm2}, we have the following estimates.\\
a) \begin{equation}
   (w_{\delta }(t),z_{\delta}(t))_{2} \geqslant \alpha_{3}\left |w(t) \right |_{2}^{2}+\gamma_{\delta}(t),\notag 
\end{equation}
where $ \gamma_{\delta}(t) =(w_{\delta }(t),z_{\delta}(t))_{2}-(w(t),z(t))_{2}$, $\alpha^{-1}_{3}=\left | \frac{\beta(x,u)}{u}\right |_{L^{\infty}}+1$.\\
b) \begin{equation}
\begin{split}
   (\nabla \Psi_{\varepsilon}(z_{\delta }(t)),\zeta_{\delta}(t))_{2}-&  (\nabla \Psi_{\varepsilon}(z_{\delta}(t)),(Aw)_{\delta }(t))_{2} \leqslant \left | \nabla\Psi_{\varepsilon}(z_{\delta}(t) \right |_{2} \cdot (\left | \zeta(t) \right |_{2}+\left | (Aw)_{\delta}(t) \right |_{2}) \\ \leqslant& \left | \nabla\Psi_{\varepsilon}(z_{\delta }(t) \right |_{2} \cdot (\alpha_{4}\left | w(t)\right |_{2}+\left | A \right |_{\infty}\left | w(t)\right |_{2}) \\ \leqslant& C \left | \nabla\Psi_{\varepsilon}(z_{\delta}(t) \right |_{2} \left | w(t)\right |_{2},\notag
\end{split}   
\end{equation}
where $\alpha_{4}=\alpha_{[-\left | u \right |_{\infty},\left | u \right |_{\infty}]}$, $C=\alpha_{4}+\left | A \right |_{\infty}$.\\
c) For all $\varepsilon, \delta \in (0,1)$, we have \begin{equation}
    \varepsilon (\left|\Psi_{\varepsilon}(z_{\delta }(t))\right|,\left|w_{\delta }(t)\right|)_{2}  \leqslant \varepsilon \left | \Psi_{\varepsilon}(z_{\delta }(t)) \right |_{\infty} \left | w_{\delta}(t)  \right |_{1} \leqslant c\left | z(t) \right |_{\infty} \left | w(t)  \right |_{1}, \quad \text{a.e.} \ t \in (0,T). \notag
\end{equation}
\begin{equation}
 \displaystyle\lim_{\varepsilon \rightarrow 0} \sup_{\delta \in (0,1)}\left| \varepsilon(\Psi_{\varepsilon}(z_{\delta }(s))\right|_{\infty}=0, \quad \text{a.e.} \ s \in (0,T). \notag
\end{equation}
d) For $\delta(\varepsilon)$ as in Claim \ref{clm2}, we have
\begin{equation} 
    \displaystyle \int_{0}^{T}\left|(R_{\delta(\varepsilon) }(s),\Psi_{\varepsilon}(z_{\delta(\varepsilon) }(s)))_{2}\right|ds \leqslant \varepsilon \left| z \right|_{L^{2}(0,T;L^{2})}. \notag
\end{equation}
Therefore, we have
\begin{equation}
     0 \leqslant h_{\varepsilon, \delta(\varepsilon)}(t) \leqslant \tilde \zeta_{\varepsilon}\text{exp}(\frac{C^{2}}{\alpha_{3}\gamma}t), \quad \forall t \in [0,T], \notag
\end{equation}
where $\tilde \zeta_{\varepsilon}$ is as in the proof of Claim \ref{clm2}. This implies that $h_{\varepsilon,\delta(\varepsilon)}\rightarrow 0$ as $\varepsilon \rightarrow 0$ uniformly on $[0,T]$, and hence $\nu_{1}\equiv \nu_{2} $, as claimed.

 $\hfill\square $   
\end{proof}
\begin{corollary}\label{col2}
Assume that Hypothesis (\MakeUppercase{\romannumeral 1}) holds. Let $T>0$ and $\nu_{0} \in \mathcal{P}\cap L^{\infty} $ and let $u \in L^{1}((0,T)\times \mathbb{R}^{d})\cap L^{\infty}((0,T)\times \mathbb{R}^{d})$ be a distributional solution to (\ref{1}). If $\nu_{1}, \nu_{2} \in L^{1}((0,T)\times \mathbb{R}^{d})\cap L^{\infty}((0,T)\times \mathbb{R}^{d})$ are two nonnegative distributional solutions to (\ref{10}), then $\nu_{1} \equiv \nu_{2}$.      
\end{corollary}
\begin{proof}
\rm This follows immediately from Theorem \ref{thm4.1} by the same method as in the proof of Corollary \ref{col1}.
 $\hfill\square $    
\end{proof}
The following result then standardly follows (see \cite[Thm.~4.1]{ref4}). For the readers' convenience we include the proof.
\begin{theorem}\label{thm4.2}
Assume that Hypothesis (\MakeUppercase{\romannumeral 1}) holds. Let $T>0$, and let $X^{1}(t)$ and $X^{2}(t)$ be two probabilistically weak solutions to (\ref{9}) on the stochastic bases $(\Omega_{1} ,\mathscr{F}_{1},(\mathscr{F}^{1}_{t})_{t \geqslant 0}, \mathbb{P}_{1})$, $(\Omega_{2} ,\mathscr{F}_{2},(\mathscr{F}^{2}_{t})_{t \geqslant 0}, \mathbb{P}_{2})$ such that $\mathcal{L}_{X^{1}(t)}, \mathcal{L}_{X^{2}(t)}$ are absolutely continuous w.r.t Lebesgue measure $dx$ and for
\begin{equation}
    u_{1}(t,\cdot):=\frac{d\mathcal{L}_{X^{1}(t)}}{dx},\quad u_{2}(t,\cdot):=\frac{d\mathcal{L}_{X^{2}(t)}}{dx},\notag
\end{equation}
we have
 \begin{equation}
    u_{1},u_{2} \in L^{1}((0,T)\times \mathbb{R}^{d})\cap L^{\infty}((0,T)\times \mathbb{R}^{d}) .\notag
 \end{equation}
 Then, $X^{1}$ and $X^{2}$ have the same law, i.e. $\mathbb{P}_{1}\circ (X^{1})^{-1}=\mathbb{P}_{2}\circ (X^{2})^{-1}$.
\end{theorem}
\begin{proof}
 \rm   By narrow continuity, we have 
 \begin{equation}
     u(0,x)dx=u_{0}(dx)=(\mathbb{P}_{1}\circ X^{1}(0)^{-1})(dx)=(\mathbb{P}_{2}\circ X^{2}(0)^{-1})(dx),\notag
 \end{equation}
 \begin{equation}
    u_{1}(t,\cdot), u_{2}(t,\cdot) \in L^{\infty}(\mathbb{R}^{d}), \; \text{for all} \; t \in (0,T] .\notag
 \end{equation}
By It\^o's formula, $u_{1}$ and $u_{2}$ are probability solutions to the nonlinear Fokker-Planck equation (\ref{0}), hence, by Corollary \ref{col1}, $u_{1} \equiv u_{2}$. Again by It\^o's formula, both marginal laws $\mathbb{P}_{1}\circ (X^{1})^{-1}$ and $\mathbb{P}_{2}\circ (X^{2})^{-1} $ satisfy the martingale problem with initial condition $u_{0}$ for the linear Kolmogorov operator
\begin{equation}
   L_{u}:= \sum\limits_{i,j=1}^{d}\frac{a_{ij}(x)\beta(x,u)}{u}\frac{\partial^{2}}{\partial x_{i}\partial x_{j}}+b(x, u)\cdot \nabla.\notag
\end{equation}
Hence, by Corollary \ref{col2}, the assertion follows from Lemma 2.12 in \cite{ref25}. More specifically, for $s \in [0,t]$, choose the set $\mathcal{R}_{[s,T]}$ appearing in Lemma 2.12 to be the set of all weakly continuous probability solutions $\nu$ of (\ref{10}) such that, for each $t \in [s,T]$, $t>0$, $\nu(t,\cdot) \in L^{\infty}(\mathbb{R}^{d}) \cap L^{1}(\mathbb{R}^{d})$ and $\nu \in L^{1}((0,T)\times \mathbb{R}^{d})\cap L^{\infty}((0,T)\times \mathbb{R}^{d})$. Therefore, (2.9) and (2.14) in \cite{ref25} hold. By Corollary \ref{col2}, condition \romannumeral1) in Lemma 2.12 of \cite{ref25} holds, and hence condition \romannumeral2) in Lemma 2.12 of \cite{ref25} follows; that is, $X^{1}$ and $X^{2}$ have the same law.
 
 $\hfill\square $     
\end{proof}
\subsection{Strong solvability of the corresponding McKean-Vlasov SDEs}
\begin{theorem}\label{thm4.3}
Let $T>0$ and $d \ne 2$, assume that Hypotheses (\textbf{\romannumeral1}'),(\textbf{\romannumeral2}), (\textbf{\romannumeral3}'), (\textbf{\romannumeral5}) hold. Let $u_{0} \in L^{\infty}$ be a probability density on $\mathbb{R}^{d}$, and let $\sigma (x) = a(x)^{1/2}$ in (\ref{9}). Then, the probabilistically weak solution to (\ref{9}) from Section \ref{section4.1} is also the unique strong solution to (\ref{9}).    
\end{theorem}
\begin{proof}
 \rm By Theorem \ref{thm3.1} and Theorem \ref{thm3.2}, there exists a weakly continuous probability solution $u \in L^{1}((0,T); L^{1})\cap L^{\infty}((0,T)\times \mathbb{R}^{d})$ to (\ref{1}). Moreover, by the proof of Theorem \ref{thm3.2} in the Appendix, for $\lambda \in (0,\lambda_{0}^{*}]$, $f \in L^{1}\cap L^{2}$, we have $u(\lambda,f):=J_{\lambda}(f) \in H^{1}$. Consider the finite-difference approximation in Definition \ref{def2.1}
 \begin{equation}\label{step}
    u_{h}^{k+1}+hA(u_{h}^{k+1})=u_{h}^{k}, \ \forall k=0,1, \cdots ,N_{h}; \ u_{h}^{0}=u_{0}.
\end{equation}
Analogously to the discrete energy argument in the proof of Theorem 5.2 in \cite{ref7}, in particular (5.15)-(5.17) therein, testing (\ref{step}) by
$u_{h}^{k+1}$, our conditions imply that, for some constant $C>0$ 
\begin{equation}
\left| u_{h}^{k+1} \right|_{2}^{2}-  \left| u_{h}^{k} \right|_{2}^{2}+\gamma \delta h \left| \nabla u_{h}^{k+1} \right|_{2}^{2} \leqslant Ch\left| u_{h}^{k+1} \right|_{2}^{2}  .\notag
\end{equation}
By the $L^{1}$-contraction property and estimate (\ref{5'''}), we have
\begin{equation}
 \sup_{kh \leqslant T} \left| u_{h}^{k} \right|_{2}^{2} \leqslant \sup_{kh \leqslant T} \left| u_{h}^{k} \right|_{1}\cdot \left| u_{h}^{k} \right|_{\infty}   .\notag
\end{equation}
 Consequently, summing the preceding inequality over $k$ yields, for some constant $C_{T}>0$
 \begin{equation}
   h \sum_{k=0}^{N_{h}-1} \left| \nabla u_{h}^{k+1} \right|_{2}^{2} \leqslant C_{T}   .\notag
 \end{equation}
 Thus, the corresponding step functions are bounded in $L^{2}(0,T;H^{1})$. Since they converge to the mild solution $u$ in $L^{1}$ uniformly on $[0,T]$ and are uniformly bounded in $L^{\infty}((0,T)\times \mathbb{R}^{d})$, they converge to $u$ also in $L^{2}(0,T;L^{2})$. Moreover, by weak lower-semicontinuity,
 \begin{equation}
   \int_{0}^{T} \left|\nabla u(t) \right|_{2}^{2}dt\leqslant \liminf _{h\rightarrow 0} \int_{0}^{T} \left|\nabla u_{h}^{k+1} \right|_{2}^{2}dt \leqslant C_{T} .\notag
 \end{equation}
 Hence, $u \in L^{2}(0,T;H^{1})$.\\
 Therefore, there exists a weakly continuous probability solution $u$ to (\ref{1}) such that 
 \begin{equation}
  u \in L^{2}((0,T); H^{1})\cap L^{\infty}((0,T)\times \mathbb{R}^{d})   .\notag
 \end{equation}
Set 
\begin{equation}
  F(t,x):=b(x, u(t,x)), \quad \Sigma(t,x):=\sqrt{2\Phi(x, u(t,x))}S(x)  ,\notag
\end{equation}
where $S(x)=a^{1/2}(x)$, $a(x)=(a_{ij}(x))_{1\leqslant i,j \leqslant d}$, and $\Phi$ is as defined in (\textbf{\romannumeral1}'). Our conditions imply that $S \in C^{1}_{b}(\mathbb{R}^{d};\mathbb{R}^{d \times d})$. More precisely, by (\textbf{\romannumeral2}), we have $a(x) \geqslant \gamma I $ for some $\gamma>0$, where $I$ is an identity matrix. Hence, $S(x) \geqslant \sqrt{\gamma} I$. Since the square-root map is smooth on the open cone of positive-definite symmetric matrices, $S$ is differentiable. Differentiating $S(x)^{2}=a(x)$ with respect to $x_{l}$ gives the Sylvester equation
\begin{equation}
 S(x)\partial_{l}S(x)+(\partial_{l}S(x))S(x)=\partial_{l}a(x)   .\notag
\end{equation}
At a fixed $x$, choose an orthogonal matrix $Q$ such that
\begin{equation}
  Q^{T}S(x)Q=\text{diag}(s_{1}, \dots , s_{d}), \quad s_{i} \geqslant \sqrt{\gamma}, \; i,j=1, \dots, d. \notag
\end{equation}
Set
\begin{equation}
X:=Q^{T}\partial_{l}S(x)Q, \quad H:=Q^{T}\partial_{l}a(x)Q.  \notag
\end{equation}
Then the Sylvester equation becomes
\begin{equation}
  (s_{i}+s_{j})X_{ij}=H_{ij}, \; i,j=1, \dots d  .\notag
\end{equation}
Consequently, 
\begin{equation}
 \left| X_{ij}\right| \leqslant \frac{\left| H_{ij} \right|}{2 \sqrt{\gamma}},  \; i,j=1, \dots d,  \notag
\end{equation}
hence, in Hilbert–Schmidt norm, we have
\begin{equation}
  \left| \partial_{l}a(x)^{1/2} \right|_{HS}  \leqslant \frac{1}{2\sqrt{\gamma}} \left| \partial_{l}a(x) \right|_{HS} .\notag
\end{equation}
Therefore,
\begin{equation}
   \left| \nabla S(x)\right|_{\infty} \leqslant \frac{1}{2\sqrt{\gamma}} \left| \nabla a(x)\right|_{\infty} < \infty .\notag
\end{equation}
By (\textbf{\romannumeral1}'), 
 (\textbf{\romannumeral2}) and (\textbf{\romannumeral3}'), we have 
 \begin{equation}\label{strong1}
   F \in L^{\infty}((0,T)\times \mathbb{R}^{d}; \mathbb{R}^{d}), \quad \Sigma \in  L^{\infty}((0,T)\times \mathbb{R}^{d}; \mathbb{R}^{d\times d})  .
 \end{equation}
Since $\nabla_{x} F^{j}=b^{j}_{x}(x,u)+b^{j}_{r}(x,u)\nabla u$, $i,j=1, \dots, d$ and (\textbf{\romannumeral3}') implies that $b_{x}(x,r)$ and $b_{r}(x,r)$ are locally bounded, combined with $u \in L^{2}((0,T);H^{1})$, we have 
\begin{equation}\label{strong2}
  \nabla F \in L^{2}((0,T); L^{2}_{\text{loc}}(\mathbb{R}^{d};\mathbb{R}^{d\times d}))  .
\end{equation}
Since $\nabla_{x} \Sigma = q_{x}(x,u)S(x)+q_{r}(x,u)\nabla u S(x)+q(x,u)\nabla S(x)$, where $q(x,u)=\sqrt{2\Phi(x,u)}$. Hence, by  (\textbf{\romannumeral1}'), 
 (\textbf{\romannumeral2}), (\textbf{\romannumeral3}') and $u \in L^{2}((0,T);H^{1})$, we have 
 \begin{equation}\label{strong3}
    \nabla \Sigma \in L^{2}((0,T); L^{2}_{\text{loc}}(\mathbb{R}^{d};\mathbb{R}^{d\times d \times d})) .
 \end{equation}
 On the one hand, as shown above, by the superposition principle, there exists a probabilistically weak solution to (\ref{9}). On the other hand, by (\ref{strong1}), (\ref{strong2}), (\ref{strong3}) and \cite[Thm.~4.4 and Rem.~4.5]{ref32} (see also \cite[Thm.~2.5.1]{ref33}), it follows that for (\ref{9}) we have pathwise uniqueness (see, \cite[Def.~2.4]{ref32}). Therefore, the assertion follows from \cite[Thm.~3.3]{ref32} (see also \cite[Thm.~3.4.1]{ref33}).

$\hfill\square $   
\end{proof}

\section{Nonlinear Markov processes}
The results in this section standardly follow from our distributional uniqueness results in Section 3 (see e.g., \cite{ref13} or \cite{ref28}). For the readers' convenience we recall the necessary definitions and present the precise arguments necessarily to derive them from the main result in \cite{ref28}.\\
\indent We first review the basic notions from \cite{ref28}. Let $\Omega_{s}:=C([s, \infty),\mathbb{R}^{d})$ denote the space of continuous paths in $\mathbb{R}^{d}$ starting at time $s$, equipped with the Borel $\sigma$-algebra $\mathcal{B}(\Omega_{s})$. For $\tau \geqslant s$, 
\begin{equation}
    \pi_{\tau}^{s}: \Omega_{s} \rightarrow \mathbb{R}^{d}, \; \;  \pi_{\tau}^{s}(\omega):=\omega (\tau), \omega \in  \Omega_{s}, \quad \mathcal{F}_{s,r}:=\sigma (\pi_{\tau}^{s}|s \leqslant \tau \leqslant r).  \notag
\end{equation}
Let $\tilde{\mathcal{P}}(\mathbb{R}^{d})$ denote the set of all Borel probability measures on $\mathbb{R}^{d}$, and 
\begin{equation}
 \mathcal{P}_{a}= \left \{ \mu \in \tilde{\mathcal{P}} | \mu \ll  dx
     \right \}, \; \; \mathcal{P}^{\infty}_{a}= \left \{ \mu \in \mathcal{P}_{a} | \frac{d\mu}{dx} \in L^{\infty}(\mathbb{R}^{d})\right \}=\mathcal{P} \cap L^{\infty} .\notag
\end{equation}
As in Section 4, for an initial condition $ (s,\zeta) \in \mathbb{R}^{+} \times \tilde{\mathcal{P}}(\mathbb{R}^{d}) $, we denote the probabilistically weak solution of (\ref{9}) by $ (X_{t}^{s,\zeta})_{t\geqslant s}$, defined on a stochastic basis $(\Omega ,\mathscr{F}, (\mathscr{F}_{t})_{t \geqslant s}, \mathbb{P}) $, and set $\mathbb{P}_{s,\zeta} := \mathbb{P}\circ (X^{s,\zeta}_{\cdot})^{-1}$. Nonlinear Markov processes are given by such families of laws $ \mathbb{P}_{s,\zeta}$, $s \geqslant 0$, on path space, where $\zeta$ ranges over all admissible initial probability measures on $\mathbb{R}^{d}$. The following is the rigorous definition.
\begin{definition}\label{def5.1}(Nonlinear Markov process \cite{ref28}) Let $\mathcal{P}_{0} \subseteq \tilde{\mathcal{P}}(\mathbb{R}^{d})$. A nonlinear Markov process is a family $(\mathbb{P}_{s, \zeta })_{(s,\zeta)\in \mathbb{R}^{+} \times \mathcal{P}_{0}}$ of probability measures $\mathbb{P}_{s,\zeta} $ on $ \mathcal{B}(\Omega_{s})$ such that\\
(\textbf{\romannumeral1}) The marginals $ \mathbb{P}_{s,\zeta} \circ (\pi^{s}_{t})^{-1} =: \mu_{t}^{s, \zeta}  $ belong to $\mathcal{P}_{0}$ for all $ 0 \leqslant s \leqslant t$, $\zeta \in \mathcal{P}_{0}$. \\
(\textbf{\romannumeral2}) The nonlinear Markov property holds, i.e., for all $ 0 \leqslant s \leqslant r \leqslant t$, $\zeta \in \mathcal{P}_{0}$, 
\begin{equation}
  \mathbb{P}_{s,\zeta}(\pi_{t}^{s} \in A |\mathcal{F}_{s,r})(\cdot)=p_{(s,\zeta),(r,\pi_{r}^{s}(\cdot))}(\pi_{t}^{r}\in A) \quad  \mathbb{P}_{s,\zeta}-\text{a.s. \; for all} \; A \in \mathcal{B}(\mathbb{R}^{d}),\notag
\end{equation}
where $p_{(s,\zeta),(r,y)}$, $y \in \mathbb{R}^{d}$, is a regular conditional probability kernel from $\mathbb{R}^{d}$ to $\mathcal{B}(\Omega_{r})$ of $\mathbb{P}_{r,\mu_{r}^{s,\zeta}}[\cdot | \pi_{r}^{r}=y]$, $y \in \mathbb{R}^{d}$ (i.e., in particular $p_{(s,\zeta),(r,y)} \in \tilde{\mathcal{P}}(\Omega_{r})$ and $p_{(s,\zeta),(r,y)}(\pi_{r}^{r}=y)=1 $). 
\end{definition}
\begin{remark}\label{rem5.1}
\rm The one-dimensional time-marginals $\mu_{t}^{s, \zeta} $ of a nonlinear Markov process satisfy the flow property, i.e.,
\begin{equation}
    \mu_{t}^{s, \zeta}=\mu_{t}^{r, \mu_{r}^{s, \zeta}} ,\; \; \forall \; 0 \leqslant s \leqslant r \leqslant t, \;\; \zeta \in \tilde{\mathcal{P}}(\mathbb{R}^{d}), \notag
 \end{equation}
which corresponds to the Chapman-Kolmogorov equations for the linear Markov process.
\end{remark}

\begin{theorem}\label{thm5.1}
 Consider NFPE (\ref{1}), assume that Hypotheses (\textbf{\romannumeral1}'),(\textbf{\romannumeral2}), (\textbf{\romannumeral3}'), (\textbf{\romannumeral4})-(\textbf{\romannumeral5}) hold. Then there exists a nonlinear Markov process $(\mathbb{P}_{s, \zeta })_{(s,\zeta)\in [0,T]\times \mathcal{P}_{a}^{\infty}}$ such that, for each $(s,\zeta)\in [0,T]\times \mathcal{P}_{a}^{\infty} $, the measure $\mathbb{P}_{s, \zeta }$ is the path law of the unique weak solution to the McKean-Vlasov SDE (\ref{9}) with initial condition $(s, \zeta)$. 
\end{theorem}
\begin{proof}
    \rm We identify a measure that is absolutely continuous with respect to Lebesgue measure $dx$ with its density, that is, we identify $u^{s,\zeta}$ with $ (\mu_{t}^{s,\zeta})_{t \geqslant s} $, where $\mu_{t}^{s,\zeta}:=u^{s,\zeta}(t,x)dx $. We first note that all previous theorems continue to hold when the nonlinear Fokker-Planck equation (\ref{1}) is considered on $[s, \infty) \times \mathbb{R}^{d}$ for any $s \geqslant 0$, with an initial condition $u_{0}$ renamed as $\zeta \in \mathcal{P}_{a}^{\infty}$. Then, by Corollary \ref{col3.2}, there exists a weakly continuous probability solution $u^{s,\zeta}(t,x)$ to (\ref{1}) from $(s, \zeta)$ such that 
    \begin{equation}
      u^{s,\zeta}(t,x) \in \bigcap_{T>s}L^{\infty}((s,T)\times \mathbb{R}^{d}) ,\quad u^{s,\zeta}(t, \cdot) \in \mathcal{P}_{a}^{\infty} ,\; \text{for all}\; t \geqslant s, \notag
    \end{equation}
and $\left \{ u^{s,\zeta}(t,x) \right \}_{(s,\zeta) \in \mathbb{R}^{+} \times \mathcal{P}_{a}^{\infty} } $ is a solution flow in $\mathcal{P}_{a}^{\infty}$. By the superposition principle, there exists a probabilistically weak solution $X_{t}^{s,\zeta}$ to (\ref{9}) with initial condition $(s,\zeta) \in [0,T] \times \mathcal{P}_{a}^{\infty} $.     \\
\indent Moreover, by Corollary \ref{col2}, the linearized equation (\ref{18}) has a unique weakly continuous probability solution in $\bigcap_{T>s}L^{\infty}((s,T)\times \mathbb{R}^{d}) $. Therefore, by Corollary 3.9 in \cite{ref28}, there exists a nonlinear Markov process $(\mathbb{P}_{s, \zeta })_{(s,\zeta)\in [0,T]\times \mathcal{P}_{a}^{\infty}}$ with one-dimensional time-marginals $\mathbb{P}_{s,\zeta} \circ (\pi^{s}_{t})^{-1} = u^{s,\zeta}(t,x)dx $. Finally, for $(s,\zeta) \in [0,T] \times \mathcal{P}_{a}^{\infty} $, by Theorem \ref{thm4.2}, $ \mathbb{P}_{s, \zeta }$ is the path law of the unique probabilistically weak solution $X^{s,\zeta}_{t}$ to the McKean-Vlasov SDE (\ref{9}).

$\hfill\square $         
\end{proof}

\section*{Appendix}\label{appendix}
\noindent \textbf{Proof of Theorem \ref{thm3.2}:}\\
\indent Here, we even prove the theorem for general diffusion terms $a_{ij}(x,u)$ satisfying the conditions (H1)-(H3) in \cite{ref13}, Section 2.1, and not only for our special case $ \displaystyle \tilde a_{ij}(x,u)=\frac{a_{ij}(x)\beta(x,u)}{u}$. We first recall the main steps of the proof of the existence of mild solutions in \cite{ref13} (also in \cite{ref4}). We define in the space $L^{1}$ the operator $A_{0}: D(A_{0}) \subset L^{1} \rightarrow L^{1}$,
\begin{equation}
\begin{split}
& A_{0}(u)=- \sum\limits_{i,j=1}^{d} D^{2}_{ij}(a_{ij}(x,u)u)+ \text{div}(b(x,u)u), \quad \forall u \in D(A_{0}) ,\\&
    D(A_{0})= \left\{ u \in L^{1}; \; A_{0}(u) \in L^{1} \right\}.\notag
\end{split}    
\end{equation}
Once the operator $A \subset A_{0}$ is shown to be $m$-accretive in $L^{1}$, the operator $A$ generates a continuous contraction semigroup $S(t)$ in $L^{1}$ given by the exponential formula
\begin{equation} \label{000}
S(t)u_{0}=\lim_{n \rightarrow \infty}\left(I+\frac{t}{n}A\right)^{-n}u_{0}, \quad u_{0} \in L^{1}, \;\forall t \geqslant 0. 
\end{equation}
Formula (\ref{000}) means that $u(t)=S(t)u_{0}$ is a mild solution to the Cauchy problem
\begin{equation}
    \begin{split}
       \displaystyle & \frac{du}{dt}+Au=0 \quad \text{on}\; (0, \infty),\\&
        u(0)=u_{0},
    \end{split}\notag
\end{equation}
and hence a mild solution to NFPE (\ref{*}). Consider the following regularized functions
\begin{flalign}
 a_{ij}^{\epsilon}(x,u):=\phi_{\epsilon}(x,u)a_{ij}(x,u)+(1-\phi_{\epsilon}(x,u))\text{max}(\gamma, \tilde{\gamma})\delta_{ij},   \notag
\end{flalign}
\begin{equation}
 b_{i}^{\epsilon}(x,u):= \phi_{\epsilon}(x,u)b_{i}(x,u)  ,\notag
\end{equation}
for $1 \leqslant i, j \leqslant d$, $x \in \mathbb{R}^{d}$, $u \in \mathbb{R}$, $\tilde{\gamma}:=d \displaystyle \sup_{1 \leqslant i,j \leqslant d} \left|a_{ij} \right|_{\infty} $ such that $ \sum\limits_{i,j=1}^{d}a_{ij}(x,u)\xi_{i}\xi_{j} \leqslant \tilde{\gamma} \left| \xi \right|^{2}$ for any $u \in \mathbb{R}$, $\xi \in \mathbb{R}^{d} $. $\phi_{\epsilon}(x,u):=\phi(\epsilon \left|u \right|)\phi(\epsilon \left|x \right|^{2})$ with $\phi \in C^{\infty}(\mathbb{R})$, $\mathds{1}_{(-\infty, 1]} \leqslant \phi \leqslant 1-\mathds{1}_{[2, \infty)}$, $-2 \leqslant \phi' \leqslant 0$. $a_{ij}^{\epsilon}$ and $b_{i}^{\epsilon}$ also satisfy the following conditions 
\begin{equation}\label{002}
    \begin{split}
 (a_{ij}^{\epsilon,*}(x,u))_{u} \in &C_{b}(\mathbb{R}^{d} \times \mathbb{R}), \quad \text{and for some} \; \tilde C \in (0, \infty), g \in L^{2}, 1 \leqslant i,j \leqslant d,\\&
 \left| (a_{ij}^{\epsilon,*})_{u}(x,u) -(a_{ij}^{\epsilon,*})_{u}(x,\bar{u})\right| \leqslant \tilde C \left| u-\bar{u}\right|,\\&
 \left| \tilde{b}^{\epsilon,*}(x,u) -\tilde{b}^{\epsilon,*}(x,\bar{u})\right| \leqslant g(x) \left| u-\bar{u}\right|, \quad \forall u, \bar{u} \in \mathbb{R}, \; x \in \mathbb{R}^{d}
    \end{split}
\end{equation}
where $a_{ij}^{\epsilon,*}(x,u)=a_{ij}^{\epsilon}(x,u)u $, $b^{\epsilon,*}(x,u)=b^{\epsilon}(x,u)u$, $ \tilde{b}^{\epsilon,*}_{i}(x,u):= b_{i}^{\epsilon,*}-\sum\limits_{j=1}^{d}(a_{ij}^{\epsilon,*})_{x_{j}}$, $\tilde{b}^{\epsilon,*}:=(\tilde{b}^{\epsilon,*}_{1}, \dots , \tilde{b}^{\epsilon,*}_{d})$.
Then for $\lambda_{0}^{*} \in (0,\frac{1}{2}\gamma (b_{\infty}+c_{\infty})^{-2})$, where 
\begin{equation}
   b_{\infty}:= \text{sup}\left\{ \left|b_{i}(x,u) \right|; (x,u) \in \mathbb{R}^{d}\times \mathbb{R}, i=1, \dots ,d \right\} ,\notag
\end{equation}
\begin{equation}
\begin{split}
   c_{\infty}:=\text{sup}&\left\{ \left|(a_{ij})_{x_{j}}(x,u) \right|; (x,u) \in \mathbb{R}^{d}\times \mathbb{R}, i,j=1, \dots ,d \right\} \\&+ 8(\displaystyle \max_{i,j}\left| a_{ij}\right|_{\infty}\cdot \max_{i,j}\left| (a_{ij})_{x_{j}}\right|_{\infty}+\max(\gamma,\tilde{\gamma})) ,\notag
\end{split}
\end{equation}
the equation 
\begin{equation}\label{003}
   u-\lambda \sum\limits_{i,j=1}^{d} D_{ij}^{2} a_{ij}^{\epsilon,*}(x,u)+ \lambda \text{div}(b^{\epsilon}(x,u)u)  =f ,
\end{equation}
has a solution $u_{\epsilon}=u_{\epsilon}(\lambda,f)$ for each $\lambda \in (0, \lambda_{0}]$ and $ \forall f \in L^{1}\cap L^{2}$. 
We also have, for $\lambda \in (0, \lambda_{0}^{*}]$, $f \in L^{1}\cap L^{2}$, along a subsequence $\left\{\epsilon \right\} \rightarrow 0$,
\begin{equation}\label{001}
\begin{split}
    u_{\epsilon }(\lambda, f)\rightarrow u=u(\lambda,f)\quad &\text{weakly in} \  H^{1}, \\& \text{strongly in} \ L_{\text{loc}}^{2}, \\& \text{weak-star in} \ L^{\infty}.
  \end{split}
\end{equation}
Since $J_{\lambda}$ is Lipschitz, it extends to all of $L^{1}$. Defining $u(\lambda,f):=J_{\lambda}(f)$ and $u_{\epsilon}(\lambda,f):=J_{\lambda}^{\epsilon}(f)$ for $f \in L^{1}$, where $J_{\lambda}^{\epsilon}: L^{1} \rightarrow D(A_{0}^{\epsilon}) \subset L^{1}$ and $(A_{0}^{\epsilon}, D(A_{0}^{\epsilon}))$ is defined as $(A_{0}, D(A_{0}))$ with $a_{ij}^{\epsilon}$ and $b_{i}^{\epsilon}$ replacing $a_{ij}$ and $b_{i}$, respectively, we have 
\begin{equation}\label{004}
    J_{\lambda}^{\epsilon}(f) \rightarrow J_{\lambda}(f) \quad \text{strongly in} \;  L^{1} .
\end{equation}
We now continue the proof. For $f \in L^{1}\cap L^{\infty}$, take $M = \displaystyle \frac{\lambda}{\lambda_{0}} \left | f\right|_{\infty}$, where $\lambda_{0}=\displaystyle \text{min}\left\{\frac{\gamma}{2(b_{\infty}+c_{\infty})}, \frac{1}{2C_{0}} \right\}$, $\lambda \in (0, \lambda_{0})$, $C_{0}=2c_{d}\left| \delta \right|_{\infty}$, $\delta(r)=\text{max}\left (\delta_{1}(r),\delta_{2}(r)\right)$, where $\delta_{1}$ is as defined in (\ref{5'}) and $\delta_{2}$ is as defined in (\ref{5''}) with $a_{ij}$ replacing $\tilde a_{ij}$ , and $c_{d}$ is a constant depending only on the dimension $d$. By (\ref{003}), we have 
\begin{equation}\label{005}
  \begin{split}
u_{\epsilon}&- \left | f\right|_{\infty}-M -\lambda \sum\limits_{i,j=1}^{d} D^{2}_{ij}(a^{\epsilon,*}_{ij}(x,u_{\epsilon})-a^{\epsilon,*}_{ij}(x,\left | f\right|_{\infty}+M))\\&+ \lambda \text{div}\left(b^{\epsilon}(x, u_{\epsilon})u_{\epsilon}-b^{\epsilon}(x, \left | f\right|_{\infty}+M)(\left | f\right|_{\infty}+M)\right)=f- \left | f\right|_{\infty}-M\\& -\lambda (\left | f\right|_{\infty}+M) [\text{div}\left(b^{\epsilon}(x, \left | f\right|_{\infty}+M) \right)-\sum\limits_{i,j=1}^{d} D^{2}_{ij}(a^{\epsilon,*}_{ij}(x,\left | f\right|_{\infty}+M) )].
  \end{split}  
\end{equation}
Since
\begin{equation}
    \begin{split}
 -\lambda (\left | f\right|_{\infty}+M) & [\text{div}\left(b^{\epsilon}(x, \left | f\right|_{\infty}+M) \right)-\sum\limits_{i,j=1}^{d} D^{2}_{ij}(a^{\epsilon,*}_{ij}(x,\left | f\right|_{\infty}+M) )]\\& \leqslant  c_{d}(\delta_{1} (\left | f\right|_{\infty}+M )+\delta_{2} (\left | f\right|_{\infty}+M ))\cdot \frac{ \gamma (\left | f\right|_{\infty}+M )}{2(b_{\infty}+c_{\infty})^{2}}
 \\ & \leqslant  \lambda C_{0} (\left | f\right|_{\infty}+M) \\& \leqslant M,   
    \end{split}
\end{equation}
the right-hand side of (\ref{005}) is non-positive. Thus, multiplying (\ref{005}) by $\mathcal{X}_{\delta}((u_{\epsilon}-\left | f\right|_{\infty}-M)^{+}) $ and integrating over $\mathbb{R}^{d}$, by (\ref{002}), we obtain

\begin{align}  
 \begin{split}
&\int_{\mathbb{R}^{d}} (u_{\epsilon}-\left | f\right|_{\infty}-M) \mathcal{X}_{\delta}((u_{\epsilon}-\left | f\right|_{\infty}-M)^{+})dx \\& \leqslant \int_{\mathbb{R}^{d}}\lambda \sum\limits_{i,j=1}^{d} D^{2}_{ij}(a^{\epsilon,*}_{ij}(x,u_{\epsilon})-a^{\epsilon,*}_{ij}(x,\left | f\right|_{\infty}+M)) \cdot \mathcal{X}_{\delta}((u_{\epsilon}-\left | f\right|_{\infty}-M)^{+})dx\\&-\int_{\mathbb{R}^{d}} \lambda \text{div}\left(b^{\epsilon}(x, u_{\epsilon})u_{\epsilon}-b^{\epsilon}(x, \left | f\right|_{\infty}+M)(\left | f\right|_{\infty}+M)\right) \cdot \mathcal{X}_{\delta}((u_{\epsilon}-\left | f\right|_{\infty}-M)^{+})dx \\&= -  \int_{\mathbb{R}^{d}}\lambda \sum\limits_{i,j=1}^{d} D_{i}(a^{\epsilon,*}_{ij}(x,u_{\epsilon})-a^{\epsilon,*}_{ij}(x,\left | f\right|_{\infty}+M)) \cdot D_{j}\mathcal{X}_{\delta}((u_{\epsilon}-\left | f\right|_{\infty}-M)^{+})dx  \\&+ \int_{\mathbb{R}^{d}} \lambda \left(b^{\epsilon}(x, u_{\epsilon})u_{\epsilon}-b^{\epsilon}(x, \left | f\right|_{\infty}+M)(\left | f\right|_{\infty}+M)\right) \cdot \nabla \mathcal{X}_{\delta}((u_{\epsilon}-\left | f\right|_{\infty}-M)^{+})dx \\&= \int_{\mathbb{R}^{d}} \lambda \left(\tilde{b}^{\epsilon,*}(x, u_{\epsilon})-\tilde{b}^{\epsilon,*}(x, \left | f\right|_{\infty}+M)\right) \cdot \nabla \mathcal{X}_{\delta}((u_{\epsilon}-\left | f\right|_{\infty}-M)^{+})dx\\&= \frac{\lambda}{\delta}\int_{[(u_{\epsilon}-\left | f\right|_{\infty}-M)^{+}\leqslant \delta]}  \left(\tilde{b}^{\epsilon,*}(x, u_{\epsilon})-\tilde{b}^{\epsilon,*}(x, \left | f\right|_{\infty}+M)\right) \cdot \nabla u_{\epsilon}dx.
    \end{split}\notag
\end{align}
Since
\begin{equation}
\begin{split}
    \lim_{\delta \rightarrow 0} &\frac{\lambda}{\delta}\int_{[(u_{\epsilon}-\left | f\right|_{\infty}-M)^{+}\leqslant \delta]} \left| \left(\tilde{b}^{\epsilon,*}(x, u_{\epsilon})-\tilde{b}^{\epsilon,*}(x, \left | f\right|_{\infty}+M)\right) \cdot \nabla u_{\epsilon} \right| dx \\& \leqslant \lambda \left|g \right|_{2}  \lim_{\delta \rightarrow 0} \left(\int_{[(u_{\epsilon}-\left | f\right|_{\infty}-M)^{+}\leqslant \delta]} \left| \nabla u_{\epsilon} \right| ^{2}dx \right )^{\frac{1}{2}}=0,
\end{split}    \notag
\end{equation}
we have  
\begin{equation}
  \left| u_{\epsilon}\right|_{\infty} \leqslant (1+\frac{\lambda}{\lambda_{0}}) \left | f\right|_{\infty} .\notag
\end{equation}
Thus, by (\ref{001}) and (\ref{004}), we obtain 
\begin{equation}
    \left| J_{\lambda}(u)\right|_{\infty} \leqslant (1+\frac{\lambda}{\lambda_{0}}) \left | u\right|_{\infty} ,\quad \forall u \in L^{1} \cap L^{\infty},\notag
\end{equation}
take $C=(1/\lambda_{0})$, by (\ref{000}), the proof is complete.
$\hfill\square $

\vspace{1em}
\noindent \textbf{Acknowledgements.} This work was funded by the Deutsche Forschungsgemeinschaft
(DFG, German Research Foundation) - Project-ID 317210226 - SFB 1283. The second named author also gratefully acknowledges financial support by the DFG - IRTG 2235 - Project-ID 282638148.

\end{document}